\documentclass[11pt, leqno]{article}

\usepackage{amsmath,amsthm,amssymb,enumerate}
\usepackage{epsfig}
\usepackage{dina42e}
\usepackage{amssymb}
\usepackage{amsmath,xspace}
\usepackage{amssymb}
\usepackage{amsmath,amsthm}
\usepackage[latin1]{inputenc}
\usepackage[T1]{fontenc}
\usepackage{path}
\usepackage{ae,aecompl}
\usepackage{amsfonts}
\usepackage{amsxtra}
\usepackage{bbm,euscript,mathrsfs}
\usepackage{color}

\allowdisplaybreaks

\usepackage{amsfonts}
\usepackage{amsxtra}

\usepackage[frak,hyperref,theorem]{paper_diening}

\usepackage{color}

\newcommand{\Div}{\divergence}
\newcommand{\eps}{\varepsilon}
\newcommand{\R}{\mathbb R}
\newcommand{\N}{\mathbb N}

\newcommand{\Q}{\mathcal Q}

\newcommand{\LL}{\mathcal L}

\newcommand{\dd}{\mathrm{d}}
\newcommand{\dx}{\,\mathrm{d}x}
\newcommand{\dt}{\,\mathrm{d}t}
\newcommand{\dxt}{\,\mathrm{d}(x,t)}
\newcommand{\ds}{\,\mathrm{d}\sigma}
\newcommand{\dxs}{\,\mathrm{d}(x,\sigma)}
\newcommand{\A}{\mathcal{A}}
\newcommand{\B}{\mathcal{B}}
\newcommand{\dom}{\operatorname{dom}}

\DeclareMathOperator{\Bog}{Bog}

\newcommand{\correct}[1]{{\color{red}#1}}

\title{Weak Solutions for a Non-Newtonian Diffuse Interface Model with Different Densities}
  
\author{
Helmut Abels\\
 {\it Department of Mathematics, University of Regensburg}\\ {\it  Universit\"atsstra\ss e 31, 93053  Regensburg, Germany}
\\[1ex]
  Dominic Breit\\
 {\it Department of Mathematics, Heriot-Watt University}\\ {\it Edinburgh EH14 4AS, UK}\\
}

\begin{document}
\maketitle

\date{}

\begin{abstract}
We consider weak solutions for a diffuse interface model of two non-Newtonian viscous, incompressible fluids of power-law type in the case of different densities in a bounded, sufficiently smooth domain. This leads to a coupled system of a nonhomogenouos generalized Navier-Stokes system and a Cahn-Hilliard equation. For the Cahn-Hilliard part a smooth free energy density and a constant, positive mobility is assumed.  Using the $L^\infty$-truncation method we prove existence of weak solutions for a power-law exponent $p>\frac{2d+2}{d+2}$, $d=2,3$. 
\end{abstract}





\section{Introduction}

We consider a two-phase flow of two incompressible non-Newtonian fluids in a bounded domain. In the case of classical sharp interface models the interface, separating the two fluids, is modeled as a (sufficiently smooth) hypersurface. But these models do not allow to describe flows beyond the occurrence of topological singularities, e.g. when droplets collide or pinch off. In the following we will consider a diffuse interface model for such two-phase flows, where the fluids are assumed to be partly miscible and the sharp interface is replaced by a thin interfacial region, where a scalar order parameter $\varphi$ changes smoothly, but rapidely between two distinguished values, e.g. $\pm 1$, that describe the separate phases.
More precisely we consider the following system, which couples a nonhomogeneous generalized Navier-Stokes system and a Cahn-Hilliard equation:
\begin{alignat}{2}\label{eq:NSCH1}
      \varrho\partial_t \bfv+(\varrho\bfv+\bfJ))\cdot \nabla \bfv-\Div \bfS(\varphi,\bfD\bfv)+\nabla \pi&=-\varepsilon_0\Div \big(\nabla \varphi\otimes\nabla \varphi\big),\\\label{eq:NSCH2}
      \Div \bfv&=0,\\\label{eq:NSCH3}
      \partial_t \varphi+\bfv\cdot\nabla \varphi&=m\Delta \mu,\\\label{eq:NSCH4}
      \mu&=\varepsilon^{-1}f'(\varphi)-\varepsilon_0\Delta \varphi
\end{alignat}
in a space-time cylinder $\Q_T=\Omega\times (0,T)$, where $\bfJ= - \frac{\partial \rho}{\partial \varphi} \nabla \mu$ and $\Omega\subseteq \R^d$, $d=2,3$, is a bounded domain with $C^4$-boundary together with suitable boundary and initial conditions specified below.
 Here $\bfv\colon \Q_T\to \R^d$ is the velocity of the mixture of the two non-Newtonian fluids, which is defined as the volume average of the individual fluid velocities, cf.~\cite{AbelsGarckeGruen2}, $\pi\colon\Q_T\to \R$ is its pressure, $\varphi \colon \Q_T\to \R$ is an order parameter related to the volume fractions of the fluids (e.g. $\varphi$ is the difference of the volume fractions of both fluids), $f\colon \R\to \R$ is the homogeneous free energy density of the fluid mixture, $\rho=\rho (\varphi)$ is the density of the mixture, depending explicitely on $\varphi$, and $\mu\colon \Q_T\to \R$ is the chemical potential of the mixture. Moreover, $\bfS(\varphi,\bfD\bfv)$ is the viscous part of the stress tensor due to friction in the fluid mixture, which will be specified below. Finally, $\eps_0>0$ is a constant related to the thickness of the diffuse interface and $m>0$ is a mobility coefficient, which is assumed to be constant as well. For simplicity we will assume $\eps_0=m=1$. But all results in the following remain valid for general $\eps_0,m>0$.

The model above is a non-Newtonian variant of the model derived in \cite{AbelsGarckeGruen2}, where the constitutive assumption $\bfS(\varphi, \bfD\bfv)= 2\nu(\varphi)\bfD\bfv$ is made. A prototypical example for the following is 
\begin{equation}\label{eq:0209}
  \bfS(\varphi, \bfD\bfv)= 2\nu(\varphi)|\bfD \bfv|^{p-2}\bfD \bfv\qquad \text{for some }p> 1
\end{equation}
and a suitable positive $\nu \colon \R\to \R$. Such power-law models for non-Newtonian fluids
are very popular among rheologists \cite{AstaritaMarucci,BirdArmstrongHassager}.
The value for $p$ can be specified by physical experiments. An extensive list for specific values for different fluids can be found in \cite{BirdArmstrongHassager}. Apparently many interesting $p$-values lie in the interval $[\frac 32,2]$. The mathematical discussion of power-law models for non-Newtonian fluids started in the late sixties with the work of Lions and Ladyshenskaya (see \cite{La1,La2,Ladyzenskaja2} and \cite{Lions}). A first systematic study can be found in \cite{MNRR}.

In the case of \eqref{eq:0209} the derivation of \eqref{eq:NSCH1}--\eqref{eq:NSCH4} is precisely the same as in \cite[Section~2]{AbelsGarckeGruen2}. One just has to choose $\bfS(\varphi, \bfD\bfv)$ as above, which guarantees the validity of the local dissipation inequality. Moreover, let us note that in the derivation of the model in \cite{AbelsGarckeGruen2} it is assumed that
\begin{equation}\label{eq:rho}
  \varrho(\varphi)= \frac{\tilde{\varrho}_2-\tilde{\varrho}_1}2\varphi + \frac{\tilde{\varrho}_2+\tilde{\varrho}_1}2,
\end{equation}
where $\tilde{\varrho}_1,\tilde{\varrho}_2$ are the specific densities of the two (separate) fluids. Then $\frac{\partial \varrho}{\partial \varphi}(\varphi)$ is constant and $\varrho(\varphi)$ solves the continuity equation
\begin{equation}\label{eq:contEq}
  \partial_t \varrho(\varphi) + \Div (\varrho(\varphi) \bfv + \bfJ)=0,
\end{equation}
where $\bfJ$ is a flux of the fluid density due to diffusion relative to the flux $\rho \bfv$ caused by convection. Here $\varphi= \varphi_2-\varphi_1$ is the difference of the volume fractions $\varphi_1,\varphi_2$ of the fluids. Physically $\varphi$ should only attain values in $[-1,1]$, which guarantees that $\varrho(\varphi)$ is positive. Since the density is a function of the order parameter, the study share certain similar features with the analysis of quasi-compressible fluids, see for instance \cite{Fe} and the references therein. We note that this diffuse interface model corresponds to a two-phase flow with a sharp interface $\Gamma(t)$ separating two immiscible incompressible non-Newtonian fluids. Here the surface $\Gamma(t)$ gives rise to a surface energy given by a constant surface tension coefficient times the area of $\Gamma(t)$. Note that no variable surface tension or curvature effects are taken into account in the surface energy. We refer to \cite[Section~4]{AbelsGarckeGruen2} for the relation to sharp interface models in the Newtonian case, which can be modified to the present situation.\\
In the following we want to construct weak solutions of the system above for arbitrary large times $0<T<\infty$. But, since a comparison principle for the fourth order Cahn-Hilliard system \eqref{eq:NSCH3}--\eqref{eq:NSCH4} is unknown and we will assume that $f\colon\R\to \R$ is a suitable smooth function, we are not able to prove that $\varphi$ attains only values in $[-1,1]$. Let us note that in the case of Newtonian fluids (i.e., $p=2$) the existence of weak solutions of the system above for large times was proven in \cite{AbelsDepnerGarcke,AbelsDepnerGarckeDegMob}. In these contributions either $f\colon [-1,1]\to\R$ is assumed to be singular at $\partial[-1,1]$ or the mobility $m$ is a degenerate function of $\varphi$. In both cases one obtains that $\varphi \in [-1,1]$ almost everywhere and one can assume that \eqref{eq:rho} holds true. Since in our setting we are not able to show $\varphi\in [-1,1]$ and $\rho(\varphi)$ in \eqref{eq:rho} becomes negative outside of $[-1,1]$ unless $\tilde{\rho}_1=\tilde{\rho}_2$, we cannot assume \eqref{eq:rho} for all values of $\varphi$ and $\rho(\varphi)$ has to be modified outside of $[-1,1]$ such that it stays strictly positive. But then $\frac{\partial \rho}{\partial \varphi}$ is no longer constant for all values of $\varphi$ and $\varrho$ solves only
\begin{equation}\label{eq:contEq'}
  \partial_t \varrho + \Div (\varrho \bfv + \bfJ)=R, \quad \text{where}\quad R= -\nabla \frac{\partial \varrho}{\partial \varphi}\cdot \nabla \mu,
\end{equation}
instead of \eqref{eq:contEq}. Here $R$ is an additional source term, which vanishes in the interior of  $\{\varphi \in [-1,1]\}$ if \eqref{eq:rho} holds for all $\varphi\in [-1,1]$. In order to obtain a local dissipation inequality and global energy estimate the equation of linear momentum \eqref{eq:NSCH1} has to be modified to 
\begin{equation}\label{eq:NSCH1'}
        \varrho\partial_t \bfv+(\varrho\bfv+\bfJ)\cdot \nabla \bfv+ R\frac{\bfv}2-\Div \bfS(\varphi,\bfD\bfv)+\nabla \pi=-\varepsilon_0\Div \big(\nabla \varphi\otimes\nabla \varphi\big).
\end{equation} 
This modification guarantees that
\begin{equation*}
  \partial_t \frac{\varrho |\bfv|^2}2 + \Div \left((\varrho \bfv+\bfJ)  \frac{ |\bfv|^2}2\right)- \Div (\bfS(\varphi,\bfD\bfv) \bfv - \pi \bfv)- \bfS(\varphi,\bfD\bfv) :\bfD\bfv = -\eps_0 \Div (\nabla \varphi\otimes \nabla \varphi) \cdot \bfv
\end{equation*}
and the (global) energy identity
\begin{equation*}
  \frac{\dd}{\dt} \left(\int_\Omega\frac{\varrho |\bfv|^2}2\dx + \int_\Omega \left(\frac{\eps_0 |\nabla \varphi|^2}2 + \frac{f(\varphi)}\eps_0 \right)\dx \right) = - \int_\Omega  \bfS(\varphi,\bfD\bfv) :\bfD\bfv\dx - \int_\Omega m|\nabla \mu|^2 \dx
\end{equation*}
for every sufficiently smooth solution of
\begin{align}\label{eq:NSCH}
  \left\{\begin{array}{rl}      
      \varrho\partial_t \bfv+(\varrho\bfv+\bfJ))\cdot \nabla \bfv+ R\frac{\bfv}2-\Div \bfS(\varphi,\bfD\bfv)+\nabla \pi&=-\varepsilon_0\Div \big(\nabla \varphi\otimes\nabla \varphi\big),\\
      \Div \bfv&=0,\\
      \partial_t \varphi+\bfv\cdot\nabla \varphi&=m\Delta \mu,\\
      \mu&=\varepsilon_0^{-1}f'(\varphi)-\varepsilon_0\Delta \varphi,
      \end{array}\right.
\end{align}
where $\bfJ= - \frac{\partial \rho}{\partial \varphi} \nabla \mu$,
together with the boundary and initial conditions
\begin{align}\label{eq:1.2'}
\left\{\begin{array}{rl}
\bfv|_{\partial\Omega}=0&\qquad \mbox{ \,on $\partial \Omega\times (0,T)$,}\\
\partial_{\mathcal N} \varphi|_{\partial\Omega}=\partial_{\mathcal N}\mu|_{\partial\Omega}=0&\qquad \mbox{ \,on $\partial \Omega \times (0,T)$,}\\
\bfv(0,\cdot)=\bfv_0&\qquad \mbox{ \,in $\Omega$,}\\
\qquad \varphi(0,\cdot)=\varphi_0&\qquad\mbox{ \,in $\Omega$.}\end{array}\right.
\end{align}
Here $\mathcal N$ denotes the exterior normal of $\partial\Omega$. Finally, we note that \eqref{eq:NSCH1'} is equivalent to 
\begin{equation}\label{eq:NSCH1''}
        \partial_t (\varrho \bfv)+\Div (\bfv\otimes (\varrho\bfv+\bfJ))- R\frac{\bfv}2-\Div \bfS(\varphi,\bfD\bfv)+\nabla \pi=-\varepsilon_0\Div \big(\nabla \varphi\otimes\nabla \varphi\big)
\end{equation}
due to \eqref{eq:contEq'}. This reformulation will be used for the definition of weak solutions below.

In the following we will prove existence of weak solutions for \eqref{eq:NSCH}--\eqref{eq:1.2'}. So far \eqref{eq:NSCH}--\eqref{eq:1.2'} was only treated in the case that both fluids have the same densities, i.e., $\rho(\varphi)\equiv const.$. First analytic results in this case were obtained by Kim, Consiglieri, and Rodrigues~\cite{NonNewtonianModelH}. They proved existence of weak solutions if $p\geq \frac{3d+2}{d+2}$, $d=2,3$. In this case monotone operator techniques can be applied. In~\cite{GrasselliPrazakNonNewtonianDIM} Grasselli and Pra\v{z}\'ak discussed the longtime behavior of solutions of the system in the case $p\geq \frac{3d+2}{d+2}$, $d=2,3$, assuming periodic boundary conditions and a regular free energy density. For the same $p$ results on existence of weak solutions with a singular free energy density $f$ and the longtime behavior  were obtained by Bosia~\cite{Bosia} in the case of a bounded domain in $\R^3$. 
Finally, existence of weak solutions was shown by Abels, Diening, and Terasawa~\cite{NonNewtonianModelH} in the case that $p>\frac{2d}{d+2}$ using the parabolic Lipschitz truncation method for divergence free vector fields developed by Breit, Diening, and Schwarzacher~\cite{BreitDieningSchwarzacher}, which is the same range for $p$ as for a single power-law type fluid, cf. Diening, R{\r u}{\v{z}}i{\v{c}}ka, and Wolf~\cite{DieningRuzickaWolf} (the same bound appears in stationary results \cite{FrehseMalekSteinhauer2,DieningMalekSteinhauer}). For reference on analytic results in the Newtonian case ($p=2$) we refer to the introduction of \cite{AbelsDepnerGarcke}. 

Unfortunately, the Lipschitz truncation method of \cite{BreitDieningSchwarzacher} is not applicable to  \eqref{eq:NSCH} since the system provides only control of $\partial_t (\varrho \bfv)$ and not of $\partial_t \bfv$ (unless $\rho\equiv const.$) in suitable spaces. Alternatively we will use the $L^\infty$-truncation method, which was already successfully applied in \cite{Bu,Fr,Wolf} to prove existence of weak solutions for power-law type fluids if $p >\frac{2d+2}{d+2}$. 

Throughout the paper we make the following assumptions:
\begin{enumerate}
\item[(A1)] $f\colon\R\to \R$ is three-times differentiable such that there is some $C>0$ satisfying
  \begin{equation*}
    |f'''(s)|\leq C(|s|+1)\qquad \text{for all}\quad s\in\R
  \end{equation*}
  and $f''(s)\geq - \alpha$ for all $s\in\R$ and some $\alpha \geq 0$.
\item[(A2)]  $\bfS\colon \R\times \R_{sym}^{d\times d}\to \R^{d\times d}$ satisfies 
\begin{align*}
  |\bfS(s,\bfM)|&\leq C(|\bfM|^{p-1}+1)\\ 
  |\bfS(s_1,\bfM)-\bfS(s_2,\bfM)|&\leq C|s_1-s_2|(|\bfM|^{p-1}+1)\\ 
  \bfS(s,\bfM):\bfM&\geq \omega |\bfM|^p -C_1
\end{align*}
for all $\bfM\in \R_{sym}^{d\times d}$, $s,s_1,s_2\in \R$, and some $C,C_1,\omega>0$, $p\in (\frac{2d+2}{d+2},\infty)$. Here  $\bfA:\bfB= \operatorname{tr}(\bfA^T\bfB)$ and  $\R^{d\times d}_{sym}=\{A\in\R^{d\times d}: A^T=A\}$.
Moreover, we assume that $\bfS(c,\cdot)\colon \R^{d\times d}_{sym}\to \R^{d\times d}_{sym}$ is strictly monotone for every $c\in \R$, i.e. $\bfM_1,\bfM_2\in \R_{sym}^{d\times d}$
\begin{align*}
\big(\bfS(c,\bfM_1)-\bfS(c,\bfM_2)\big):\big(\bfM_1-\bfM_2\big)\geq0
\end{align*}
with $``=``$ if and only if $\bfM_1=\bfM_2$.
\item[(A3)] $\varrho\colon \R\to \R$ is twice continuously differentiable and strictly positive. Moreover, $\varrho,\varrho',\varrho''$ are bounded.   
\end{enumerate}
We note that (A1) implies that 
\begin{equation*}
  |f(s)|\leq C(|s|^4+1), \qquad 
  |f'(s)|\leq C(|s|^3+1), \qquad 
  |f''(s)|\leq C(|s|^2+1)
\end{equation*}
for all $s\in\R$ and some $C>0$.

Our main result result is:
\begin{theorem}\label{thm:main}
  Let $\Omega\subseteq \R^d$ be a bounded domain with $C^4$-boundary, $0<T<\infty$, and let (A1)--(A3) be valid. Then for every $\bfv_0\in L^2_\sigma(\Omega)$, $\varphi_0\in B^3_{2,4}(\Omega)$ with $\partial_{\mathcal N} \varphi_0|_{\partial\Omega}=0$, there is a weak solution 
  \begin{equation*}
    \bfv\in L^\infty(0,T;L^2_\sigma(\Omega))\cap L^p(0,T;W^{1,p}_0(\Omega)^d), \varphi \in W^{1,4}(0,T;L^2(\Omega))\cap L^4(0,T;W^{4,2}(\Omega))
  \end{equation*}
  of \eqref{eq:NSCH}--\eqref{eq:1.2'} in the sense that 
  \begin{alignat*}{1}
    &-\int_{Q_T} \varrho(\varphi) \bfv\cdot \partial_t \bfeta \,\dd (x,t)- \int_\Omega \varrho(\varphi_0)\bfv_0\cdot \bfeta|_{t=0}\dx
    - \int_{Q_T} \bfv\otimes (\varrho(\varphi) \bfv + \bfJ):\nabla \bfeta\, \dd(x,t)\\
&\quad - \int_{Q_T} R \frac{\bfv}2\cdot \bfeta\, \dd(x,t) +\int_{Q_T} \bfS(\varphi,\bfD\bfv): \bfD\bfeta\, \dd(x,t) 
= \eps_0 \int_{Q_T} \nabla \varphi\otimes \nabla \varphi: \nabla \bfeta \, \dd(x,t)
  \end{alignat*}
  for all $\bfeta\in C^\infty([0,T)\times\Omega)^d$ with $\Div \bfeta=0$, \eqref{eq:NSCH}$_3$--\eqref{eq:NSCH}$_4$ are satisfied pointwise almost everywhere and \eqref{eq:1.2'}$_1$, \eqref{eq:1.2'}$_2$, \eqref{eq:1.2'}$_4$  hold true in the sense of traces.
\end{theorem}
The structure of the article is as follows: In Section~\ref{sec:Approx} we prove existence of weak solutions of a suitable approximation of \eqref{eq:NSCH}--\eqref{eq:1.2'}, where the convection terms $\rho \bfv \cdot  \nabla\bfv$ and $\bfv\cdot \nabla \varphi$ are mollified. This is done with the aid of a (partial) Galerkin approximation for the Navier-Stokes part using unique solvability of \eqref{eq:NSCH}$_3$--\eqref{eq:NSCH}$_4$ for given $\bfv$. Then in Section~\ref{sec:MainResult} we prove Theorem~\ref{thm:main} with the aid of the $L^\infty$-truncation method by passing to the limit in the approximated system.

\section{Notation and Preliminaries}

For $a,b \in \R^d$ we denote $a\otimes b = ab^T=(a_ib_j)_{i,j=1}^d\in\R^{d\times d}$.
The standard Lebesgue spaces (with respect to the Lebesgue measure) are denoted by $L^p(M)$, $1\leq p\leq \infty$, for some measurable $M\subset\R^N$, $L^p(M;X)$ the $X$-valued Bochner space, and $L^p(0,T;X)=L^p((0,T);X)$. 
 The standard $L^p$-Sobolev space is denoted by $W^{m,p}(\Omega)$. $W^{m,p}_0(\Omega)$ is the closure of $C_0^\infty(\Omega)$ in $W^{m,p}(\Omega)$ and $H^m(\Omega)=W^{m,2}(\Omega), H^m_0(\Omega)=W^{m,2}_0(\Omega)$.  Finally, spaces of divergence free test functions will be denoted with a subscript ``$\sigma$''. In particular, $C_{0,\sigma}^\infty(\Omega)=\{\bfvarphi\in C_0^\infty(\Omega)^d:\Div \bfvarphi=0\}$. $L^2_\sigma(\Omega)$ is the closure of $C_{0,\sigma}^\infty(\Omega)$ in $L^2(\Omega)$. Finally, $L_0^q(\Omega)$ denotes the subspace of $L^q(\Omega)$ of functions wit zero mean value.\\
By $B^s_{p,q}(\R^d)$, $s\in\R$, $1\leq p,q\leq \infty$, we denote the standard Besov space and $B^s_{p,q}(\Omega)$ its restriction to $\Omega$. We note that all results on interpolation of Sobolev- and Besov spaces for $\R^d$ carry over to the spaces on $\Omega$ using a suitable extension operator, cf. e.g. the discussion in \cite[Section 2]{ModelH}. In particular, we have for every $k\in\N$, $1\leq p_0,p_1<\infty$ that
\begin{equation*}
  (W^{k,p_0}(\Omega), W^{k+1,p_1}(\Omega))_{\theta,p}= B^{k+\theta}_{pp}(\Omega)\qquad \frac1p=\frac{1-\theta}{p_0}+\frac{\theta}{p_1},k\in \N_0,  
\end{equation*}
for all $\theta\in (0,1)$, cf. \cite[Section~2.4.2 Theorem~1]{Triebel1}. Here $(.,.)_{\theta,p}$ denotes the real interpolation functor.

Finally, the space of bounded and uniformly continuous functions $f\colon [0,T]\to X$ for some Banach space $X$ is denoted by $BUC([0,T];X)$. The space of weakly continuous functions $f\colon [0,T]\to X$ is denoted by $C_w([0,T];X)$.

\section{The Approximated System}\label{sec:Approx}
In this section we consider the approximate system
\begin{alignat}{1}\label{Approx1.1}
      \varrho\partial_t \bfv+(\varrho(\Psi_\eps\bfv)+\bfJ))\cdot \nabla \bfv+ R\frac{\bfv}2-\Div \bfS(\varphi,\bfD\bfv)+\nabla \pi&=-\Psi_\eps\Div \big(\nabla \varphi\otimes\nabla \varphi\big),\\\label{Approx1.1b}
      \Div \bfv&=0,\\\label{Approx1.1c}
      \partial_t \varphi+(\Psi_\eps \bfv)\cdot\nabla \varphi&=\Delta \mu,\\\label{Approx1.1d}
      \mu&=f'(\varphi)-\Delta \varphi
\end{alignat}
together with 
\begin{align}\label{Approx1.2}
\left\{\begin{array}{rc}
\bfv=0\qquad\,\,\,\,\quad\quad& \mbox{ \,on $\partial \Omega\times (0,T)$,}\\
\partial_{\mathcal N} \varphi=\partial_{\mathcal N}\mu=0\qquad\,\,\,\,\quad\quad& \mbox{ \,on $\partial \Omega\times (0,T)$,}\\
\bfv(0,\cdot)=\bfv_0\,\qquad \varphi(0,\cdot)=\varphi_0&\mbox{ \,in $\Omega$,}\end{array}\right.
\end{align}
where $R=-\nabla \frac{\partial \varrho}{\partial\varphi}\cdot \nabla \mu$, $\bfJ= -\frac{\partial \varrho}{\partial\varphi} \nabla \mu$, and $\Psi_\eps= e^{-\eps A}P_\sigma$, where $A=-P_\sigma \Delta$ is the Stokes operator. Then $\Psi_\eps \bfv\in H^2(\Omega)^d\cap H^1_0(\Omega)^d\cap L^2_\sigma(\Omega)$ for all $\bfv \in  L^2_\sigma(\Omega)$ and $\eps>0$. Moreover, $\Psi_\eps \bff\to_{\eps \to 0} P_\sigma \bff$ strongly in $L^2(\Omega)^d$ for all $\bff\in L^2(\Omega)^d$.  We note that \eqref{Approx1.1c} implies
\begin{equation*}
  \partial_t \varrho +\Div (\varrho \Psi_\eps \bfv + \bfJ) = R.
\end{equation*}
Therefore \eqref{Approx1.1} is equivalent to
\begin{equation}\label{eq:35}
    \partial_t(\varrho \bfv)+\Div (\bfv \otimes (\varrho(\Psi_\eps\bfv)+\bfJ))- R\tfrac{\bfv}2-\Div \bfS(\varphi,\bfD\bfv)+\nabla \pi=-\Psi_\eps\Div \big(\nabla \varphi\otimes\nabla \varphi\big)
\end{equation}
Moreover, using
\begin{equation*}
  -\operatorname{div} (\nabla \varphi \otimes \nabla \varphi)= \mu \nabla \varphi -\nabla \left( \tfrac{|\nabla \varphi|^2}2 + f(\varphi)\right),
\end{equation*}
we can replace $-\operatorname{div} (\nabla \varphi \otimes \nabla \varphi)$ on the right hand side of \eqref{Approx1.1} by $\mu\nabla \varphi$, which will be used in the following.

In order to show the existence of a solution of \eqref{Approx1.1}--\eqref{Approx1.1d} we first solve the Navier-Stokes and Cahn-Hilliard part separately. For this we need Besov spaces. However, they will not appear
anymore in the remainder of the paper.
\begin{theorem}\label{thm:ExistenceCH}
  Let $T>0$ and $\bfv \in L^\infty(0,T;L^2_\sigma(\Omega))$ be given. Then for every $\varphi_0\in B^3_{2,4}(\Omega)$ with $\partial_{\mathcal{N}} \varphi_0=0$ there is a unique $\varphi \in W^{1,4}(0,T;L^2(\Omega))\cap L^4(0,T;W^{2,2}(\Omega))$, which solves \eqref{Approx1.1c}--\eqref{Approx1.1d} together with \eqref{Approx1.2}$_2$ and $\varphi|_{t=0}= \varphi_0$. Moreover, the mapping
  \begin{equation*}
    L^\infty(0,T;L^2_\sigma(\Omega))\ni \bfv \mapsto \mathcal S[\bfv]=\varphi \in W^{1,4}(0,T;L^2(\Omega))\cap L^4(0,T;W^{2,2}(\Omega))
  \end{equation*}
  is continuous and bounded.
\end{theorem}
\begin{remark}\label{cor:ExistenceCH}
We note that the condition $\varphi_0\in B^3_{2,4}(\Omega)$ is necessary for the regularity of $\varphi$ in the theorem because of 
\begin{align*}
 L^4(0,T; &W^{4,2}(\Omega))\cap W^{1,4}(0,T; L^2(\Omega))\hookrightarrow BUC([0,T]; (L^2(\Omega), W^{4,2}(\Omega))_{3/4,4}),\\
&(L^2(\Omega), W^{4,2}(\Omega))_{3/4,4}=B^3_{2,4}(\Omega),
\end{align*}
cf. Amann~\cite[Chapter III, Theorem 4.10.2]{Amann}. Moreover, because of $B^3_{2,4}(\Omega)\hookrightarrow\hookrightarrow C^1(\overline{\Omega})$, 
we have $\varphi\in L^\infty(0,T;C^1(\overline\Omega))$.
\end{remark}
For the following we define
\begin{alignat*}{1}
  L^{p}_{(0)}(\Omega) &= \left\{u\in L^{p}(\Omega): \int_\Omega u\dx =0\right\}\qquad \text{for}\quad1\leq p \leq \infty,\\
 H^1_{(0)}(\Omega)&= H^1(\Omega)\cap L^2_{(0)}(\Omega),\quad  H^{-1}_{(0)}(\Omega)= (H^1_{(0)}(\Omega))',
\end{alignat*}
and $\Delta_N\colon H^1_{(0)}(\Omega)\to H^{-1}_{(0)}(\Omega)$ by
\begin{equation*}
  -\left\langle\Delta_N \varphi,\psi\right\rangle_{H^{-1}_{(0)},H^{1}_{(0)}} = \int_\Omega \nabla \varphi\cdot \nabla \psi \dx \quad \text{for all}\quad \varphi,\psi\in H^{1}_{(0)}(\Omega).
\end{equation*}
As usual we identify $u\in L^2_{(0)}(\Omega)$ with $\psi\mapsto \int_\Omega u\psi \, dx\in H^{-1}_{(0)}(\Omega)$. Moreover, let $P_0 g= g- \frac1{|\Omega|}\int_\Omega g\dx$. Then $P_0\colon L^2(\Omega)\to L^2(\Omega)$ is the orthogonal projection onto $L^2_{(0)}(\Omega)$.

For the proof of Theorem~\ref{thm:ExistenceCH} we will use:
\begin{lemma}\label{lem:ExistenceCH}
  Let $T>0$ and $\bfv\in L^\infty(0,T;L^2_\sigma(\Omega))$ be given.
  Then for every $\varphi_0\in H^1(\Omega)$ there is a unique solution $$\varphi \in L^\infty(0,T;H^1(\Omega))\cap L^2(0,T;H^3(\Omega))\cap W^{1,2}(0,T;H^{-1}_{(0)}(\Omega))$$ of \eqref{Approx1.1c}--\eqref{Approx1.1d} together with \eqref{Approx1.2}$_2$ and $\varphi|_{t=0}= \varphi_0$ in the sense that
  \begin{equation*}
    \partial_t \varphi+ \Psi_\eps \bfv \cdot \nabla \varphi  = \Delta_N \mu_0(t,x)\quad \text{in }\mathcal{D}'(0,T;H^1_{(0)}(\Omega)),
  \end{equation*}
  where $\partial_{\mathcal{N}}\varphi|_{\partial\Omega} =0$ and
  \begin{equation*}
    \mu_0 = -\Delta \varphi + P_0f'(\varphi).
  \end{equation*}
\end{lemma}
\begin{proof}
In the following we assume without loss of generality that $\int_\Omega \varphi_0 \dx =0$. Otherwise we replace $\varphi$ and $f$ by $\varphi -m$ and $f(\cdot +m)$ for $m= \frac1{|\Omega|}\int_\Omega \varphi_0(x)\dx$.

The lemma will be a consequence of \cite[Theorem~4]{ModelH} and elliptic regularity theory. We refer to the Appendix~\ref{app:EvolEq} for a summary of this result and basic facts from the theory of monotone operators. To this end we define
$  H_0 =  H^{-1}_{(0)}(\Omega)$, $H_1=  H^{1}_{(0)}(\Omega)$ and $\mathcal{A}\colon \mathcal{D}(A)\to H_0$, $\mathcal{B}\colon [0,T]\times H_1\to H_0$ by
\begin{alignat*}{1}
\mathcal{A}(\varphi) &= -\Delta_N(-\Delta \varphi +f'_0(\varphi)),\quad \varphi\in\mathcal{D}(A),\\
\mathcal{D}(A) &= \{\varphi \in H^3(\Omega)\cap H^{1}_{(0)}(\Omega): \partial_{\mathcal{N}}\varphi|_{\partial\Omega} =0 \},\\
\mathcal{B}(t) (\varphi)&= -\Psi_\eps \bfv (t) \cdot \nabla \varphi -\alpha\Delta_N \varphi \quad \text{for all}\quad\varphi\in H_1, t\in [0,T],
\end{alignat*}
where $f(s)=f_0(s)-\alpha \frac{s^2}2$ with $f_0\colon \R\to\R$ convex and $\alpha\geq 0$.
  Moreover, let $E_0\colon L^2_{(0)}(\Omega)\to (-\infty,+\infty]$ be defined by $E_0(\varphi)= \int_\Omega \frac{|\nabla \varphi|^2}2 + f_0(\varphi)\dx$ if $\varphi\in \operatorname{dom}(E_0):= H^1_{(0)}(\Omega) $ and $E_0(\varphi)=+\infty$ else. 
Because of Corollary~\ref{cor:Subgradient} in the appendix,
 $\mathcal{A}= \partial_{H^{-1}_{(0)}} E_0$.
Now \cite[Theorem~4]{ModelH}, cf. Theorem~\ref{thm:MonotoneLipschitz} in the appendix, implies the existence of a unique solution $\varphi \in L^\infty(0,T;H^1(\Omega))\cap W^{1,2}(0,T;H^{-1}_{(0)}(\Omega)$ with $\varphi(t) \in\mathcal{D}(\A) $ for almost every $t\in (0,T)$ since
\begin{equation*}
  \|\mathcal{B}(t)\|_{\mathcal{L}(H_1,H_0)}\leq \|\Psi_\eps \bfv\|_{L^\infty(0,T;H^2)}+ \alpha\leq C_\eps \|\bfv\|_{L^\infty(0,T;L^2)}+\alpha
\end{equation*}
for all most every $t\in[0,T]$. Using elliptic regularity theory and the equation, one obtains additionally $\varphi \in L^2(0,T;H^3(\Omega))$.
\end{proof}
\noindent
\emph{Proof of Theorem~\ref{thm:ExistenceCH}:} Due to Lemma~\ref{lem:ExistenceCH} it remains to prove higher regularity of the solution $\varphi$ . To this end we use that
  \begin{equation*}
    \partial_t \varphi + \Delta_N\Delta\varphi =  P_0(f''(\varphi)\Delta \varphi + f'''(\varphi)|\nabla\varphi|^2)- \Psi_\eps \bfv \cdot \nabla \varphi=:F,
  \end{equation*}
  in $L^2(0,T;H^{-1}_{(0)}(\Omega))$, 
  where
  \begin{alignat*}{1}
    \|f''(\varphi)\Delta \varphi \|_{L^2(\Q_T)}&\leq \|f''(\varphi)\|_{L^\infty(0,T;L^3)}\|\Delta \varphi\|_{L^2(0,T;L^6)}<\infty\\
    \|f'''(\varphi)|\nabla \varphi|^2 \|_{L^2(\Q_T)}&\leq \|f'''(\varphi)\|_{L^\infty(0,T;L^6)}\|\nabla \varphi\|_{L^4(0,T;L^6)}^2\\
&\leq C(\|\varphi\|_{L^\infty(0,T;H^1)}+1)\|\varphi\|_{L^4(0,T;H^2)}<\infty\\
\| \Psi_\eps \bfv \cdot \nabla \varphi\|_{L^2(\Q_T)}&\leq \|\Psi_\eps \bfv\|_{L^\infty(\Q_T)}\|\nabla \varphi\|_{L^2(\Q_T)}\leq C_\eps \|\bfv\|_{L^\infty(0,T;L^2)}\|\varphi\|_{L^2(0,T;H^1)}<\infty
  \end{alignat*}
since $|f''(s)|\leq C(s^2+1)$, $|f'''(s)|\leq C(|s|+1)$ and $\varphi \in L^\infty(0,T;H^1(\Omega))\cap L^2(0,T;H^3(\Omega))\hookrightarrow L^4(0,T;H^2(\Omega))$. Hence, by linear theory we obtain $\varphi \in L^2(0,T;H^4(\Omega))\cap W^{1,2}(0,T;L^2(\Omega))$ due to \cite[Theorem~2.3]{DenkHieberPruessMZ}. 
In particular $\varphi \in L^\infty(0,T;H^2(\Omega))$. Using similar estimates as above we obtain $F\in L^4(0,T;L^2(\Omega))$. Therefore applying \cite[Theorem~2.3]{DenkHieberPruessMZ} once more, we obtain $\varphi \in W^{1,4}(0,T;L^2(\Omega))\cap L^4(0,T;H^4(\Omega))$.
\qed
\label{sec:pappr}

\begin{lemma}\label{lem:S}
Let $\kappa>0$ be arbitrary. Then the solution operator $\mathcal S[\bfv]$ from Theorem \ref{thm:ExistenceCH} satisfies
\begin{align*}
\sup_{(0,T)}\big\|\mathcal S[\bfv_1]-\mathcal S[\bfv_2]\big\|_{2,\Omega}\leq\,c(\kappa,\varepsilon,T)T^{\frac12}\sup_{(0,T)}\big\|\bfv_1-\bfv_2\big\|_{2,\Omega}
\end{align*}
for any $\bfv_1,\bfv_2$ belonging to the set
\begin{align*}
M_\kappa=\big\{\bfu\in L^\infty(0,T;L^{2}_\sigma(\Omega)):\|\bfu\|_{L^\infty(0,T;L^2(\Omega))}\leq\kappa\big\}.
\end{align*}
Here $c(\kappa, T)$ is non-decreasing with respect to $T>0$.
\end{lemma}
\begin{proof}
Let $\bfv_1,\bfv_2\in M_\kappa$ be given. We set $\varphi_1=\mathcal S[\bfv_1]$ and $\varphi_2=\mathcal S[\bfv_2]$ such that Theorem \ref{thm:ExistenceCH} implies
\begin{align}\label{eq:varphi12}
\|\varphi_1\|_{L^\infty (\Q_T)}+\|\varphi_2\|_{L^\infty(\Q_T)}\leq\,c(\kappa).
\end{align}
First of all we have 
\begin{align*}
 \partial_t (\varphi_1-\varphi_2)+(\Psi_\eps \bfv_1)\cdot\nabla \varphi_1-(\Psi_\eps \bfv_2)\cdot\nabla \varphi_2=\Delta \big(f'(\varphi_1)-f'(\varphi_2)\big)-\Delta^2 (\varphi_1-\varphi_2).
\end{align*} 
Testing with $\varphi_1-\varphi_2$ we obtain
\begin{align*}
 \frac{1}{2}\frac{\dd}{\dt}\int_\Omega&|\varphi_1-\varphi_2|^2\dx+\int_\Omega|\Delta(\varphi_1-\varphi_2)|^2\dx\\&=-\int_\Omega\Big((\Psi_\eps \bfv_1)\cdot\nabla \varphi_1-(\Psi_\eps \bfv_2)\cdot\nabla \varphi_2\Big)\,(\varphi_1-\varphi_2)\dx\\&+\int_\Omega\big(f'(\varphi_1)-f'(\varphi_2)\big)\cdot\Delta(\varphi_1-\varphi_2)\dx=:J_1+J_2.
\end{align*}
 The first term we estimate via
\begin{align*}
J_1&=\int_\Omega\Big((\Psi_\eps \bfv_1)\,\varphi_1-(\Psi_\eps \bfv_2)\, \varphi_2\Big)\cdot\nabla(\varphi_1-\varphi_2)\dx\\
&\leq\,\int_\Omega\Big|(\Psi_\eps \bfv_1)\,\varphi_1-(\Psi_\eps \bfv_2)\, \varphi_2\Big|^2\dx+\int_\Omega |\nabla(\varphi_1-\varphi_2)|^2\dx\\
&\leq\,\int_\Omega\Big|(\Psi_\eps \bfv_1)\,\varphi_1-(\Psi_\eps \bfv_2)\, \varphi_2\Big|^2\dx+c(\delta)\int_\Omega |\varphi_1-\varphi_2|^2\dx+\delta\int_\Omega |\Delta(\varphi_1-\varphi_2)|^2\dx\\
&=:J_1^1+J_1^2+J_1^3,
\end{align*}
where $\delta>0$ is arbitrary.
On account of $\bfv_1\in M_\kappa$, \eqref{eq:varphi12} and the continuity of $\Psi_\varepsilon$ we have 
\begin{align*}
J_1^1&\leq 2\int_\Omega\big|(\Psi_\eps \bfv_1)\,(\varphi_1-\varphi_2)\big|^2\dx+2\int_\Omega|\Psi_\eps (\bfv_2-\bfv_1)\, \varphi_2\big|^2\dx\\
&\leq 2\|\Psi_\eps \bfv_1\|_\infty^2\int_\Omega\big|\varphi_1-\varphi_2\big|^2\dx+2\|\varphi_2\|_\infty\int_\Omega|\Psi_\eps (\bfv_2-\bfv_1)\big|^2\dx\\
&\leq \,c(\kappa,\varepsilon)\int_\Omega\big|\varphi_1-\varphi_2\big|^2\dx+c(\kappa)\int_\Omega|\bfv_2-\bfv_1\big|^2\dx.
\end{align*}
Moreover, we obtain by (A3) and Young's inequality
\begin{align*}
J_2&=\int_\Omega \int_0^1f''\big(\varphi_2+\theta(\varphi_1-\varphi_2)\big)\,\dd\theta(\varphi_1-\varphi_2)\,\Delta(\varphi_1-\varphi_2)\dx\\
&\leq \,c(\kappa)\int_\Omega |\varphi_1-\varphi_2||\Delta(\varphi_1-\varphi_2)|\dx\\
&\leq \,c(\delta,\kappa)\int_\Omega |\varphi_1-\varphi_2|^2\dx+\delta\int_\Omega|\Delta(\varphi_1-\varphi_2)|^2\dx
\end{align*}
Plugging everything together we have shown
\begin{align*}
 \frac{1}{2}\frac{\dd}{\dt} \int_\Omega&|\varphi_1-\varphi_2|^2\dx+\int_\Omega|\Delta(\varphi_1-\varphi_2)|^2\dx\\
&\leq \,c(\delta,\kappa,\varepsilon)\int_\Omega\big|\varphi_1-\varphi_2\big|^2\dx+2\delta \int_\Omega\big|\Delta(\varphi_1-\varphi_2)\big|^2\dx+c(\kappa)\int_\Omega|\bfv_2-\bfv_1\big|^2\dx,
\end{align*}
for any $\delta>0$. Now we choose $\delta$ small enough and integrate with respect to $t$ such that (note that $\varphi_1$ and $\varphi_2$ have the same initial datum)
\begin{align*}
\int_\Omega&|\varphi_1-\varphi_2|^2\dx+\int_0^T\int_\Omega|\Delta(\varphi_1-\varphi_2)|^2\dxt\\
&\leq \,c(\kappa,\varepsilon)\bigg(\int_0^T\int_\Omega\big|\varphi_1-\varphi_2\big|^2\dxt+\int_0^T\int_\Omega|\bfv_2-\bfv_1\big|^2\dxt\bigg).
\end{align*}
The claim follows by Gronwall's lemma.
\end{proof}

\begin{theorem}
  \label{thm:ExistenceNS}
  Let Assumptions (A1)--(A4) hold true and let $\varepsilon>0$, $\bfv_0\in L^2(\Omega)$. Then there is a unique solution $(\bfv,\varphi,\mu)$ to \eqref{Approx1.1}--\eqref{Approx1.1d} and \eqref{Approx1.2}$_1$--\eqref{Approx1.2}$_3$ with $\bfv\in L^{\infty}(0,T; L^2(\Omega)) \cap L^p(0,T;W^{1,p}_{0,\sigma}(\Omega))$, $\varphi \in W^{1,4}(0,T;L^2(\Omega))\cap L^4(0,T;W^{4,2}(\Omega))$, $\mu \in L^2(0,T;W^{1,2}(\Omega))$.
\end{theorem}
\begin{proof}
The aim is to find a function
$\bfv\in L^{\infty}(0,T; L^2_\sigma(\Omega)) \cap L^p(0,T;W^{1,p}_{0,\sigma}(\Omega))$ such that
\begin{align}
\int_{\Q_T} &\varrho\bfv\cdot\partial_t\bfeta\dxt+\int_{\Q_T} \bfS(\varphi,\bfD\bfv):\bfD\bfeta\dxt\label{eq99}\\&=\int_{\Q_T}\varrho\bfv\otimes \Psi_{\varepsilon}(\bfv):\nabla\bfeta\dxt+\int_{\Q_T}\Psi_{\varepsilon}\Div(\nabla \varphi\otimes\nabla\varphi)\cdot\bfeta\dxt\nonumber\\&+\int_{\Q_T} R\tfrac{\bfv}{2}\cdot\bfeta\dxt+\int_{\Q_T}\bfv\otimes\bfJ :\nabla\bfeta\dxt+\int_\Omega \varrho_0\varphi_0\cdot\bfeta(0,\cdot)\dx\nonumber
\end{align} for all $\bfeta\in C^\infty_{0,\sigma}([0,T)\times \Omega)$, where we abbreviated $\bfJ=-\frac{\partial\varrho}{\partial\varphi}\nabla\mu$, $R=-\nabla\frac{\partial\varrho}{\partial\varphi}\cdot\nabla\mu$ and $\varrho=\varrho(\varphi)$ and $\varphi$ is the solution of \eqref{Approx1.2}$_3$--\eqref{Approx1.2}$_4$ due to Theorem~\ref{thm:ExistenceCH} and $\mu$ is defined by \eqref{Approx1.2}$_4$. Note that we used \eqref{eq:35}.\\\

\textbf{Step 1: finite dimensional approximation.}\\
We separate space and time and approximate the corresponding Sobolev space by a finite dimensional subspace which leads to a system of ODEs (Galerkin Ansatz).
To solve this we follow the approach for compressible Navier--Stokes equations introduced in \cite[Section 2]{FeNoPe}. From \cite[Appendix]{MNRR}  we infer the existence of a sequence $(\lambda_k)\subset\R$ and a sequence of functions $(\bfw_k)\subset W_{0,\sigma}^{l,2}(\Omega)$, $l\in\N$, such that
\begin{itemize}
\item[i)] $\bfw_k$ is an eigenvector to the eigenvalue $\lambda_k$ in the sense that:
$$\langle \bfw_k,\bfeta\rangle_{W_0^{l,2}(\Omega)}= \lambda_k\int_\Omega\bfw_k\cdot\bfeta\dx\quad\text{for all}\quad\bfeta\in W_{0,\sigma}^{l,2}(\Omega),$$
\item[ii)] $\int_{\Omega}\bfw_k\bfw_m\dx=\delta_{km}$ for all $k,m\in\mathbb{N}$,
\item[iii)] $1\leq\lambda_1\leq \lambda_2\leq...$ and $\lambda_k\rightarrow\infty$,
\item[iv)] $\langle\tfrac{\bfw_k}{\sqrt{\lambda_k}},\tfrac{\bfw_m}{\sqrt{\lambda_m}}\rangle_{W_0^{l,2}(\Omega)}=\delta_{km}$ for all $k,m\in\mathbb{N}$,
\item[v)] $(\bfw_k)_{k\in\N}$ is a basis of $W_{0,\sigma}^{l,2}(\Omega)$.
\end{itemize}
We choose $l>\frac{d}{2}+1$ such that $W^{l,2}_0(\Omega)\hookrightarrow W^{1,\infty}(\Omega)$
We are looking for an approximated solution $\bfv^N$ of the form
\begin{align*}
\bfv_N=\sum_{k=1}^N c_k^N\bfw_k
\end{align*}
where $\bfC^N=(c_k^N)_{k=1}^N:(0,T)\rightarrow \R^N$. So we have for a.e. $t\in(0,T)$ 
\begin{align*}
\bfv_N(t)\in X_N:=\mathrm{span}\set{\bfw_1,...,\bfw_N}.
\end{align*}
We will construct $\bfC^N$ such that $\bfv^N$ is a solution to (here we denote $\varrho_0=\varrho(\varphi_0)$)
\begin{align}
\int_\Omega\varrho(\varphi_N(t)) \bfv_N(t)\cdot\bfw_k\dx&=\int_\Omega\varrho_0 \bfv_0\cdot\bfw_k\dx+\int_{\Q_t}\bfv_N\otimes \varrho(\varphi_N)(\Psi_\eps\bfv_N+\bfJ):\nabla\bfw_k\dxs\nonumber\\\label{Gal1}&+ \int_{\Q_t} R_N\frac{\bfv_N}2\cdot\bfw_k\dxs-\int_{\Q_t} \bfS(\varphi_N,\bfD\bfv_N):\bfD\bfw_k\dxs\\\nonumber&-\int_{\Q_t}\Psi_\eps\Div \big(\nabla \varphi_N\otimes\nabla \varphi_N\big)\cdot\bfw_k\dxs \quad \text{for all}\quad t\in [0,T],
\end{align}
for $k=1,...,N$,
where 
\begin{align*}
\varphi_N&=\mathcal S[\bfv_N],\quad \bfJ_N=-\frac{\partial\varrho}{\partial\varphi}(\varphi_N)\nabla\mu_N, \quad R_N=-\nabla\frac{\partial\varrho}{\partial\varphi}(\varphi_N)\cdot\nabla\mu_N
\end{align*}
with $\mu_N=f'(\varphi_N)-\Delta \varphi_N$.
Here $\mathcal S$ denotes the solution operator from Theorem \ref{thm:ExistenceCH}.
We introduce the operator $\mathcal M_N[\varrho]:X_N\rightarrow X_N'$ with
$$\langle\mathcal M_N[\varrho]\bfu,\bfw\rangle=\int_\Omega \varrho\bfu\cdot\bfw\dx\quad \text{for all}\quad \bfu,\bfw\in X_N.$$
As $\varrho(\varphi)$ is bounded from below by some $\underline\varrho>0$ the operator $\mathcal M_N[\varrho]$ is invertible and we have
\begin{align}\label{eq:M}
\|\mathcal M_N[\varrho^1]^{-1}-\mathcal M_N[\varrho^2]^{-1}\|_{\mathcal L(X_N';X_N)}\leq\,\,c(N,\underline\varrho)\|\varrho^1-\varrho^2\|_{1}.
\end{align}
So equation (\ref{Gal1}) is equivalent to
\begin{align}\label{Gal2}
\bfv_N(t)&=\mathcal M_N\big[\varrho(\mathcal S[\bfv_N])\big]^{-1}\bigg((\varrho_0\bfv_0)'+\int_0^t\mathcal N[S[\bfv_N],\bfv_N]\ds\bigg),
\end{align}
Here $(\varrho_0\bfv_0)'$ denotes the unique $X_N'$-representative of $\varrho_0\bfv_0$ and we abbreviated
\begin{align*}
\langle\mathcal N[\varphi,\bfv_N],\bfw\rangle&=\int_\Omega\bfv_N\otimes (\varrho(\varphi)(\Psi_\eps\bfv_N)+\bfJ)):\nabla\bfw\dx+\int_\Omega R\frac{\bfv_N}2\cdot\bfw\dx\\&-\int_\Omega \bfS(\varphi,\bfD\bfv_N):\bfD\bfw\dxs-\int_\Omega\Psi_\eps\Div \big(\nabla \varphi\otimes\nabla \varphi\big)\cdot\bfw\dx,\text{ where}\\
\bfJ&=-\frac{\partial\varrho}{\partial\varphi}(\varphi)\nabla\mu, \quad R=-\nabla\frac{\partial\varrho}{\partial\varphi}(\varphi)\cdot\nabla\mu,\quad
\mu=f'(\varphi)-\Delta \varphi.
\end{align*}
As the space $X_N$ is of finite dimension it is not hard to see that a local solution $\bfv_N\in C([0,T_N];X_N)$ to \eqref{Gal2} can be found by a standard fixed point
argument provided $T_N$ is small enough (recall \eqref{eq:M} and Lemma \ref{lem:S}). Then $\mathcal P_l^N (\rho_N\bfv_N)\in C^1([0,T_N];X_N)$ as a consequence of \eqref{Gal1}. Here $P_l^N$ denotes the orthogonal projection
$W^{l,2}_{0,\sigma}(\Omega)\rightarrow X_N$.
In order to gain a global solution on $[0,T]$ we have to show that $\bfv_N$ stays bounded in $X_N$.
We differentiate \eqref{Gal1} in time and use  $\bfv_N$ as a test-function. Moreover, we use $\varphi_N=\mathcal S[\bfv_N]$ as a test-function in \eqref{Approx1.1c} and $\partial_t\varphi_N$ in \eqref{Approx1.1d}. Then we obtain by (A2)
\begin{align}\label{eq:aprioriN}
\begin{aligned}
\sup_{0\leq t\leq T_N} \bigg(\int_\Omega \varrho_N&|\bfv_N|^2+\int_\Omega\frac{|\nabla\varphi_N|^2}{2}\dx+\int_\Omega f(\varphi_N)\dx\bigg)\\&+\int_0^{T_N}\int_\Omega|\bfD\bfv_N|^p\dx\dt+\int_0^{T_N}\int_\Omega |\nabla \mu_N|^2\dx\dt\leq C(\bfv_0,\varphi_0),
\end{aligned}
\end{align}
where $C(\bfv_0,\varphi_0)$ does neither depend on $T_N$ nor on $N$. Hence we gain a global solution in time, i.e. $T_N=T$ and uniform a priori estimates in $N$. In fact, as $\varrho$
is strictly positive (uniformly in $N$) we have
\begin{align}\label{apriori1}
&\sup_{t\in(0,T)}\int_\Omega |\bfv^N(t,\cdot)|^2\dx+\int_{\Q_T}|\nabla \bfv^N|^p\dxt\leq\,C
\end{align}
for some $C>0$,
where we also used Korn's inequality. By Theorem \ref{thm:ExistenceCH} and Remark \ref{cor:ExistenceCH} there holds
\begin{align}\label{apriori2}
\varphi_N\in L^4(0,T; W^{4,2}(\Omega))\cap W^{1,4}(0,T; L^2(\Omega))
\hookrightarrow L^\infty(0,T;C^1(\overline\Omega))
\end{align}
as well as
\begin{align}\label{apriori30}
\mu_N\in L^4(0,T;W^{2,2}(\Omega)),
\end{align}
both uniformly in $N$.
For the last statement we used equation \eqref{Approx1.1c}
together with \eqref{apriori2} and $\Psi_\varepsilon\bfv_N\cdot\nabla \varphi_N\in L^\infty({\Q_T})$ (combining \eqref{apriori1} and $\Psi_\eps \colon L^2_\sigma(\Omega)\to H^2(\Omega)^d\hookrightarrow L^\infty(\Omega)^d$).\\\

\textbf{Step 3: weak convergence.}\\
On Passing to a subsequence we gain from \eqref{apriori1}--\eqref{apriori30}
\begin{align}
\bfv_N&\rightharpoonup \bfv\quad\text{in}\quad  L^p(0,T;W^{1,p}_{0,\sigma}(\Omega)),\label{konv1}\\
\bfv_N&\rightharpoonup^{*} \bfv\quad\text{in}\quad L^{\infty}(0,T;L^2(\Omega)),\label{konv2}\\
\varphi_N&\rightharpoonup \varphi\quad\text{in}\quad L^4(0,T;W^{4,2}(\Omega)),\label{eq:creg2a}\\
\varphi_N&\rightharpoonup \varphi\quad\text{in}\quad W^{1,4}(0,T;L^2(\Omega)),\label{eq:creg2b}\\
\mu_N&\rightharpoonup \mu\quad\text{in}\quad L^4(0,T;W^{2,2}(\Omega))\label{eq:creg3}
\end{align}
as $N\to\infty$.
The convergences of $\varphi_N$ above imply in particular
\begin{align}
\nabla \varphi_N&\rightharpoonup^\ast \nabla \varphi\quad\text{in}\quad L^\infty({\Q_T}),\label{eq:creg1}\\
\varphi_N&\rightarrow \varphi\quad\text{in}\quad L^q(0,T;C^1(\overline{\Omega})),\label{eq:creg2}\\
\bfJ_N&\rightarrow \bfJ \quad\text{in}\quad L^4({\Q_T}),\label{eq:cregneu}\\
R_N&\rightarrow R \quad\text{in}\quad L^4({\Q_T})\label{eq:cregneu2}
\end{align}
as $N\to \infty$,
where $q<\infty$ is arbitrary. 
Hence we can pass to the limit in the Cahn-Hilliard equation and obtain
\begin{align*}
      \partial_t \varphi+(\Psi_\eps \bfv)\cdot\nabla \varphi=\Delta \mu\qquad\text{a.e. in}\quad  \Q_T
\end{align*}
for $\mu=f'(\varphi)-\Delta \varphi$ together with 
\begin{align*}
\left\{\begin{array}{rc}
\partial_{\mathcal N} \varphi=\partial_{\mathcal N}\mu=0&\qquad\qquad \mbox{ \,on $\partial \Omega\times(0,T)$},\\
\varphi(0,\cdot)=\varphi_0&\mbox{ \,in $\Omega$.}\end{array}\right.
\end{align*}
In order to pass to the limit in the convective term we need compactness of $\bfv_N$ in $L^2(\Q_T)$.
We obtain from (\ref{Gal1})
\begin{align*}
&\int_{\Q_T} \partial_t\mathcal P_l^N(\varrho_N\bfv_N) \cdot\bfeta\dx=\int_{\Q_T} \partial_t(\varrho_N\bfv_N) \cdot \mathcal P_l^N\bfeta\dx\\&=-\int_{\Q_T} \bfS(\varphi_N,\bfD\bfv_N):\bfD\mathcal P_l^N\bfeta\dxt+\int_{\Q_T}\varrho_N\bfv_N\otimes \Psi_{\varepsilon}(\bfv_N):\nabla\mathcal P_l^N\bfeta\dxt\\&+\int_{\Q_T}\Psi_{\varepsilon}(\mu_N\nabla \varphi_N)\cdot\mathcal P_l^N\bfeta\dxt+\int_{\Q_T} R_N\tfrac{\bfv_N}{2}\cdot\mathcal P_l^N\bfeta\dxt+\int_{\Q_T}\bfv_N\otimes\bfJ_N :\nabla\mathcal P_l^N\bfeta\dxt\\
&=-\int_{\Q_T} \bfS(\varphi_N,\bfD\bfv_N):\nabla\mathcal P_l^N\bfeta\dxt+\int_{\Q_T}\varrho_N\bfv_N\otimes \Psi_{\varepsilon}(\bfv_N):\nabla\mathcal P_l^N\bfeta\dxt\\&-\int_{\Q_T}\nabla\Delta_D^{-1}\Psi_{\varepsilon}(\mu_N\nabla \varphi_N):\nabla\mathcal P_l^N\bfeta\dxt-\int_{\Q_T} \nabla\Delta_D^{-1}(R_N\tfrac{\bfv_N}{2}):\nabla\mathcal P_l^N\bfeta\dxt\\&+\int_{\Q_T}\bfv_N\otimes\bfJ_N :\nabla\mathcal P_l^N\bfeta\dxt=:\int_{\Q_T} \bfH_N:\nabla\mathcal P_l^N\bfeta\dxt
\end{align*}
for all $\bfeta \in L^p(0,T;W^{l,2}_{0,\sigma}(\Omega))$. Here $\mathcal P_l^N$ denotes the orthogonal projection into $X_N$ with respect to the $W_0^{l,2}(\Omega)$ inner product. Note that due to the choice of $\bfw_1,...,\bfw_N$ the projection $\mathcal P_l^N$ is orthogonal with respect to the $L^2(\Omega)$ inner product as well. Moreover, the operator $\Delta_D^{-1}$ is the solution operator to the Laplace equation on $\Omega$ with respect to zero Dirichlet boundary values.
On account of \eqref{apriori1}--\eqref{apriori3} we have uniformly in $N$
\begin{align}
\label{apriori4}\varrho_N\bfv_N\otimes \Psi_{\varepsilon}(\bfv_N)&\in L^\infty(0,T;L^2(\Omega)),\\
\label{apriori5}\Psi_\varepsilon(\mu_N\nabla\varphi_N)&\in L^4(0,T;L^ \infty(\Omega)),\\
\label{apriori6}R_N,\bfJ_N&\in L^4(0,T;L^{q(d)}(\Omega)),
\end{align}
where $q(d)=\frac{2d}{d-2}$ if $d\geq3$ and $q(d)<\infty$ arbitrary if $d=2$.
Here we also used the properties of $\Psi_\varepsilon$. So we have uniformly in $N$
\begin{align}\label{eq:HN}
\bfH_N \in L^{p'}(\Q_T)
\end{align}
using the properties of $\Delta_D^{-1}$ and $p>\frac{2d+2}{d+2}$.\\\

\textbf{Step 3: strong convergence.}\\
We get on account of \eqref{eq:HN} and Sobolev's embedding (recall the choice of $l$)
\begin{align}
\label{apriori3}
\begin{aligned}
&\|\partial_t\mathcal P_l^N(\varrho_N\bfv_N)\|_{L^{p'}(W^{-l,2}_{\sigma}(\Omega))}=\|\partial_t\mathcal P_l^N(\varrho_N\bfv_N)\|_{L^{p}(0,T;W^{l,2}_{0,\sigma}(\Omega))^{*}}\\
&=\sup_{\|\bfvarphi\|_{L^{p}(0,T;W^{l,2}_{0,\sigma}(\Omega))\leq 1}}\int_{\Q_T} \partial_t(\varrho_N\bfv_N)\cdot\mathcal P_l^N\bfvarphi\dxt\\
&=\sup_{\|\bfvarphi\|_{L^{p}(0,T;W^{l,2}_{0,\sigma}(\Omega))\leq 1}}\int_{\Q_T} \bfH_N: \nabla P_l^N\bfvarphi\dxt\\
&\leq\sup_{\|\bfvarphi\|_{L^{p}(0,T;W^{l,2}_{0,\sigma}(\Omega))\leq 1}}\bigg(\int_{Q_T} |\bfH_N|^{p'}\dxt\bigg)^{\frac{1}{p'}}\bigg(\int_{Q_T} |\nabla P_l^N\bfvarphi|^{p}\dxt\bigg)^{\frac{1}{p}}\\
&\leq c \sup_{\|\bfvarphi\|_{L^{p}(0,T;W^{l,2}_{0,\sigma}(\Omega))\leq 1}}\|\nabla\mathcal P_l^N\bfvarphi\|_{L^{p'}(0,T;L^{\infty}(\Omega))}\leq c.
\end{aligned}
\end{align}
Combining \eqref{konv1}, \eqref{eq:creg2} and \eqref{apriori3} with the Aubin-Lions compactness Theorem shows 
$P_l^N(\varrho_N\bfv_N)\rightarrow P_\sigma(\varrho\bfv)$ in $L^{2}(0,T;L^{2}(\Omega))$. Note that we have the pointwise convergence $\mathcal P_l^N\rightarrow P_\sigma$ in $L^2(\Omega)$.
Due to this and \eqref{konv2} we have
\begin{align*} 
\int_{\Q_T} |\sqrt{\varrho_N}\bfv_N|^2\dxt&=\int_{\Q_T} \mathcal P_l^N(\varrho_N\bfv_N)\cdot\bfv_N\dxt\to \int_{\Q_T} P_\sigma(\varrho\bfv)\cdot\bfv\dxt\\&=\int_{\Q_T} |\sqrt{\varrho}\bfv|^2\dxt
\end{align*}
for $N\rightarrow\infty$. This means that $\sqrt{\varrho_N}\bfv_N$ converges strongly in $L^2({\Q_T})$. As $\varrho_N$ is strictly positive we have compactness of $\bfv_N$ as well and hence for $N\to\infty$
\begin{align}\label{eq:convvN}
\bfv_N&\rightarrow \bfv \quad\text{in}\quad L^{s}({\Q_T})\quad \text{for all}\quad s<p\frac{d+2}{d},\\
\bfv_N&\rightarrow \bfv \quad\text{in}\quad L^{q}(0,T;L^2(\Omega))\quad \text{for all}\quad q<\infty.\label{eq:convvN2}
\end{align}
Here we have used that $L^\infty(0,T;L^2(\Omega))\cap L^p(0,T;W^1_p(\Omega))\hookrightarrow L^{p \frac{d+2}d}(\Q_T)$.
This and \eqref{eq:creg2} yield together with the continuity of $\Psi_\varepsilon$
\begin{align}
\sqrt{\varrho_N}\bfv_N&\rightarrow \sqrt{\varrho}\bfv\quad\text{in}\quad L^{s}({\Q_T}),\label{konv30}\\
\varrho_N\bfv_N\otimes\Psi_\varepsilon\bfv_N&\rightarrow \varrho\bfv\otimes\Psi_\varepsilon\bfv\quad\text{in}\quad L^{s}({\Q_T}),\label{konv3}
\end{align}
for all $s<\tfrac{d+2}{d}p$.
Due to \eqref{apriori1} and (A2) we know that $\bfS(\varphi_N,\bfD(\bfv_N))$ is bounded in $L^{p'}(\Omega)$ thus
\begin{align}
\bfS(\varphi_N,\bfD(\bfv_N))\rightharpoonup\tilde{\bfS} \quad\text{in}\quad L^{p'}(Q).\label{konv4}
\end{align}
Finally we can pass to the limit in \eqref{Gal1} such that
\begin{align}\label{eq99'}
\begin{aligned}
\int_\Omega\varrho(\varphi(t)) \bfv(t)\cdot\bfeta\dx&=\int_\Omega\varrho_0 \bfv_0\cdot\bfeta\dx+\int_0^t\int_\Omega\bfv\otimes \varrho(\Psi_\eps\bfv+\bfJ):\nabla\bfeta\dx\ds\\&+ \int_0^t\int_\Omega R\tfrac{\bfv}2\cdot\bfeta\dx\ds-\int_0^t\int_\Omega \tilde\bfS:\bfD\bfeta\dx\ds\\&-\int_0^t\int_\Omega\Psi_\eps\Div \big(\nabla \varphi\otimes\nabla \varphi\big)\cdot\bfeta\dx\ds
\end{aligned}
\end{align}
  for all $\bfeta\in C^\infty_{0,\sigma}(\Omega)$ and $t\in (0,T)$. 
This implies the following equation for $\partial_t(\varrho\bfv)$ (using continuity of $\Psi_\varepsilon$)
  \begin{align}\label{eq99''}
\begin{aligned}
-\int_{\Q_T}\varrho \bfv\cdot\partial_t\bfeta\dxt&=\int_\Omega\varrho_0 \bfv_0\cdot\bfeta(0)\dx+\int_{\Q_T}\bfv\otimes \varrho(\Psi_\eps\bfv+\bfJ):\nabla\bfeta\dxt\\&+ \int_{\Q_T} R\tfrac{\bfv}2\cdot\bfeta\dxt-\int_{\Q_T} \tilde\bfS:\bfD\bfeta\dxt\\&-\int_{\Q_T}\Psi_\eps\Div \big(\nabla \varphi\otimes\nabla \varphi\big)\cdot\bfeta\dxt
\end{aligned}
\end{align}
valid for $\bfeta\in C^\infty_{0,\sigma}([0,T)\times \Omega)$. By density argument this yields
\begin{align*}
\partial_t(\varrho\bfv)\in L^{p'}(0,T;W^{-1,p'}_{\sigma}(\Omega)).
\end{align*}
\textbf{Step 4: monotone operator theory.}\\
  We apply monotone operator theory to show 
\begin{align}
\label{eq:SSS}\tilde\bfS=\bfS(\varphi,\bfD\bfv).
\end{align}
To this end we have to study the term
  \begin{align*}
  \int_{\Q_t}&\Big(\bfS(\varphi_N,\bfD\bfv_N)-\bfS(\varphi,\bfD\bfv)\Big):\bfD\big(\bfv_N-\bfv\big)\dxs\\&=\int_{\Q_t}\Big(\tilde{\bfS}-\bfS(\varphi_N,\bfD\bfv_N)\Big):\bfD\bfv\dxs
  -\int_{\Q_t}\bfS(\varphi,\bfD\bfv):\bfD\big(\bfv_N-\bfv\big)\dxs\\
  &+\int_{\Q_t}\bfS(\varphi_N,\bfD\bfv_N):\bfD\bfv_N\dxt-\int_{\Q_t}\tilde{\bfS}:\bfD\bfv\dxs.
  \end{align*}
  We get from \eqref{konv1} and \eqref{konv4}
   \begin{align*}
&\int_{\Q_t}\Big(\tilde{\bfS}-\bfS(\varphi_N,\bfD\bfv_N)\Big):\bfD\bfv\dxs\longrightarrow0,\quad N\rightarrow\infty,\\
&\int_{\Q_t}\bfS(\varphi,\bfD\bfv):\bfD\big(\bfv_N-\bfv\big)\dxs\longrightarrow0,\quad N\rightarrow\infty
  \end{align*}
for all $t\in (0,T)$. Moreover, by \eqref{Gal1} (differentiated with respect to  $t$) and \eqref{eq99''} there holds
    \begin{align*}
\int_{\Q_t}&\bfS(\varphi_N,\bfD\bfv_N):\bfD\bfv_N\dxs-\int_{\Q_t}\tilde{\bfS}:\bfD\bfv\dxs\\
&=  \frac12\int_{\Omega}\varrho_N(0)|\bfv_N(0)|^2\dx-\frac12\int_{\Omega}\varrho_0|\bfv_0|^2\dx+\int_{\Q_t} \Big(R_N\tfrac{|\bfv_N|^2}2-R\tfrac{|\bfv|^2}2\Big)\cdot\bfeta\dxt\\
&-\int_{\Q_t}\Psi_\eps\Div \big(\nabla \varphi_N\otimes\nabla \varphi_N\big)\cdot\bfv_N\dxs+\int_{\Q_t}\Psi_\eps\Div \big(\nabla \varphi\otimes\nabla \varphi\big)\cdot\bfv\dxs\\
      &-\frac12\int_{\Omega}\varrho_N(t)|\bfv_N(t)|^2\dx+\frac12\int_{\Omega}\varrho(t)|\bfv(t)|^2\dx\\
      &=:(O)+(I)+(II)+(III).
  \end{align*}
The first term can be treated using
\begin{align*}
\int_\Omega\varrho_N(0)|\bfv_N(0)|^2\dx=\int_\Omega \mathcal P^N_l(\varrho_0\bfv_0)\cdot\mathcal M_N[\varrho_0]^{-1}(\mathcal P^N_l(\varrho_0\bfv_0)')\dx\longrightarrow\int_\Omega\varrho_0|\bfv_0|^2\dx,\quad N\rightarrow\infty.
\end{align*}
Here we used strong convergence of $\mathcal P^N_l(\varrho_0\bfv_0)$ and pointwise convergence of the operator $\mathcal M_N^{-1}[\varrho_0]$. So
  \begin{align*}
  \lim_{N\rightarrow\infty}(O)=0.
  \end{align*}
On account of \eqref{eq:cregneu2} and \eqref{eq:convvN} there holds
  \begin{align*}
  \lim_{N\rightarrow\infty}(I)=0.
  \end{align*}
  We deduce from \eqref{konv1}, \eqref{eq:creg2} and continuity of $\Psi_\varepsilon$ 
  that
  \begin{align*}
  \lim_{N\rightarrow\infty}(II)=0.
  \end{align*}
  Finally, we obtain by \eqref{konv30} for a.e. $t\in(0,T)$
  \begin{align*}
\lim_{N\rightarrow\infty}(III)=-\frac12\lim_{N\rightarrow\infty}\int_{\Omega}\varrho_N(t)|\bfv_N(t)|^2\dx+\frac12\int_{\Omega}\varrho(t)|\bfv(t)|^2\dx=0
  \end{align*}
at least after taking a subsequence.
  Hence we end up with
    \begin{align*}
  \int_{\Q_t}&\Big(\bfS(\varphi_N,\bfD\bfv_N)-\bfS(\varphi,\bfD\bfv)\Big):\bfD\big(\bfv_N-\bfv\big)\dxt\longrightarrow0,\quad N\rightarrow\infty
  \end{align*}
for almost every $t\in (0,T)$.
We have by assumption (A2) for any $\vartheta\in\big(\frac{1}{2},1\big)$
\begin{align*}
\int_{\Q_t} &\Big|\big(\bfS(\varphi_N,\bfD\bfv_N)-\bfS(\varphi,\bfD\bfv_N)\big):\bfD\big(\bfv_N-\bfv\big)\Big|^\vartheta\dxs\\&\leq\,C\,\int_{\Q_t} |\varphi_N-\varphi|^\vartheta(1+|\bfD\bfv_N|^{p-1})^\vartheta\big(|\bfD\bfv_N|+|\bfD\bfv|\big)^\vartheta\dxs\\
&\leq\,C\,\bigg(\int_{\Q_t}|\varphi_N-\varphi|^{\frac{\vartheta}{1-\vartheta}}\dxs\bigg)^{1-\vartheta}\bigg(\int_{\Q_t} (1+|\bfD\bfv_N|^{p}+|\bfD\bfv|^p)\dxs\bigg)^\vartheta
\longrightarrow0
\end{align*}
for $N\rightarrow\infty$, where we used \eqref{eq:creg2} and \eqref{apriori1}. This finally implies using monotonicity of $\bfS$
\begin{align*}
\int_{\Q_t} \Big(\big(\bfS(\varphi,\bfD\bfv_N)-\bfS(\varphi,\bfD\bfv)\big):\bfD\big(\bfv_N-\bfv\big)\Big)^\vartheta\dxs
&\longrightarrow0,\quad N\rightarrow\infty
\end{align*}
for almost every $t\in (0,T)$.
By  monotonicity of $\bfS$ (which follows from (A2)) we obtain \eqref{eq:SSS}.
\end{proof}

\section{Proof of the Main Result}\label{sec:MainResult}
\label{sec:main}

In the following let $(\bfv_n,\varphi_n,\mu_n)$ be a solution to \eqref{Approx1.1}--\eqref{Approx1.2} with $\eps=\frac1n$ which exists due to Theorem \ref{thm:ExistenceNS}. So there holds
\begin{align}
-\int_{\Q_T} &\varrho_n\bfv_n\cdot\partial_t\bfeta\dxt+\int_{\Q_T} \bfS(\varphi_n,\bfD\bfv_n):\bfD\bfeta\dxt\label{eq:weak}\\&=\int_{\Q_T}\varrho_n\bfv_n\otimes \Psi_{1/n}(\bfv_n):\bfD\bfeta\dxt+\int_{\Q_T}\Psi_{1/n}(\mu_n\nabla \varphi_n)\cdot\bfeta\dxt\nonumber\\&+\int_{\Q_T} R_n\tfrac{\bfv_n}{2}\cdot\bfeta\dxt+\int_{\Q_T}\bfv_n\otimes\bfJ_n :\bfD\bfeta\dxt+\int_\Omega \varrho_0\varphi_0\cdot\bfeta(0,\cdot)\dx\nonumber
\end{align}
for all $\bfeta\in C^\infty_{0,\sigma}([0,T)\times \Omega)$. Here we have
$\bfJ_n=-\frac{\partial\varrho}{\partial\varphi}\nabla\mu_n$, $R_n=-\nabla\frac{\partial\varrho}{\partial\varphi}(\varphi_n)\cdot\nabla\mu_n$ and $\varrho_n=\varrho(\varphi_n)$.\\\

\textbf{Step 1: weak convergence.}\\
We obtain the following a priori estimates (testing the momentum equation with $\bfv_n$, the equation for $\varphi_n$ with $\mu_n$ and the equation for $\mu_n$ with $\partial_t\varphi_n$)
\begin{align*}
\sup_{t\in(0,T)} \bigg(\int_\Omega \varrho_n&|\bfv_n|^2\dx+\int_\Omega\frac{|\nabla \varphi_n|^2}{2}\dx+\int_\Omega f(\varphi_n)\dx\bigg)\\&+\int_{\Q_T}|\nabla\bfv_n|^p\dxt+\int_{\Q_T} |\nabla \mu_n|^2\dxt\leq C.
\end{align*}
This implies the following convergences (after choosing appropriate subsequences)
\begin{align}\label{eq:conv1}
\begin{aligned}
\bfv_n&\rightharpoonup \bfv \quad\text{in}\quad L^p(0,T;W^{1,p}_{0,\sigma}(\Omega)),\\
\bfv_n&\rightharpoonup^\ast \bfv \quad\text{in}\quad L^\infty(0,T;L^2(\Omega)),\\
\bfS(\varphi_n,\bfD\bfv_n)&\rightharpoonup \tilde{\bfS} \quad\text{in}\quad L^{p'}({\Q_T}),\\
\bfv_n\otimes \varrho_n\Psi_{1/n}(\bfv_n) &\rightharpoonup\tilde{\bfH} \quad\text{in}\quad L^{p\frac{d+2}{2d}}({\Q_T})\\
\varphi_n&\rightarrow \varphi\quad\text{in}\quad L^2(0,T;C^1(\overline{\Omega})),\\
\nabla \varphi_n&\rightarrow \nabla \varphi\quad\text{in}\quad L^4({\Q_T})
\end{aligned}
\end{align}
as $n\to\infty$.
By Theorem \ref{thm:ExistenceCH} and Remark \ref{cor:ExistenceCH} we gain due to the boundedness $\bfv_n\in L^\infty(0,T;L^2_\sigma(\Omega))$
\begin{align*}
\varphi_n\in L^4(0,T; W^{4,2}(\Omega))\cap W^{1,4}(0,T; L^2(\Omega))\hookrightarrow L^\infty(0,T;C^1(\overline\Omega))
\end{align*}
uniformly in $n$ and therefore
\begin{align*}
\|\varphi_n\|_{L^2(0,T;W^{2,6})}+\|f(\varphi_n)\|_{L^2(0,T;L^6)}\leq C.
\end{align*}
So, finally we have
\begin{align}\label{eq:creg}
\begin{aligned}
\nabla \varphi_n&\rightharpoonup^\ast \nabla \varphi\quad\text{in}\quad L^\infty({\Q_T}),\\
\varphi_n&\rightarrow \varphi\quad\text{in}\quad L^q(0,T;C^1(\overline{\Omega})),\\
\varphi_n&\rightarrow \varphi\quad\text{in}\quad L^4(0,T;W^{3,2}(\overline{\Omega})),\\
\varrho(\varphi_n)&\rightarrow \varrho(\varphi)\quad\text{in}\quad L^\infty({\Q_T})
\end{aligned}
\end{align}
as $n\to\infty$
for all $q<\infty$. As we have  $\partial_t \varphi_n\in L^2({\Q_T})$ and $\bfv_n\cdot\nabla \varphi_n\in L^2({\Q_T})$ uniformly in $n$ (recall \eqref{eq:conv1}$_1$ and \eqref{eq:creg}$_2$), equation \eqref{Approx1.1c} yields
\begin{align*}
\Delta\mu_n\in L^{2}({\Q_T})
\end{align*}
uniformly in $n$. Combining this with the a priori estimates we also have
\begin{align}\label{eq:mu1}
\nabla^2\mu_n\in L^{2}({\Q_T})
\end{align}
uniformly in $n$ and hence by (A3)
\begin{align}\label{eq:J}
\nabla\bfJ_n\in L^{2}({\Q_T})
\end{align}
uniformly in $n$. 
From \eqref{eq:weak} and the a priori estimates we gain
\begin{align}\label{eq:time}
\partial_t P_\sigma(\varrho_n\bfv_n)\in L^{s_0}(0,T;W^{-1,s_0}(\Omega))
\end{align}
uniformly in $n$ for some $s_0>1$ (not that we have only control over solenoidal test-functions but $\Div(\varrho_n\bfv_n)\neq0$). In addition, we have by \eqref{eq:creg}$_1$ and the continuity of $P_\sigma$ on $W^{1,p}(\Omega)$
\begin{align}\label{eq:time}
\nabla P_\sigma(\varrho_n\bfv_n)\in L^{p}({\Q_T}).
\end{align}
Both together yield compactness of $P_\sigma(\varrho_n\bfv_n)$ in $L^p(\Q_T)$ by Aubin-Lions compactness theorem. So we have
\begin{align}\label{eq:vrho}
P_\sigma(\varrho_n\bfv_n)\rightarrow_{n\to \infty} P_\sigma(\varrho\bfv)\quad\text{in}\quad L^2({\Q_T})
\end{align}
using \eqref{eq:creg}.
Due to \eqref{eq:vrho} and \eqref{eq:conv1}$_2$ we have
\begin{align*}
\int_{\Q_T} |\sqrt{\varrho_n}\bfv_n|^2\dxt&=\int_{\Q_T} P_\sigma(\varrho_n\bfv_n)\cdot\bfv_n\dxt\longrightarrow \int_{\Q_T} P_\sigma(\varrho\bfv)\cdot\bfv\dxt\\&=\int_{\Q_T} |\sqrt{\varrho}\bfv|^2\dxt
\end{align*}
for $n\rightarrow\infty$. This implies that $\sqrt{\varrho_n}\bfv_n$ converges strongly in $L^2({\Q_T})$. As $\varrho_n$ is strictly positive we have compactness of $\bfv_n$ as well and hence
\begin{align}\label{eq:convv}
\bfv_n&\rightarrow \bfv \quad\text{in}\quad L^{s}({\Q_T})\quad \text{for all}\quad s<p\frac{d+2}{d},\\
\bfv_n&\rightarrow \bfv \quad\text{in}\quad L^{q}(0,T;L^2(\Omega))\quad \text{for all}\quad q<\infty.\label{eq:convv2}
\end{align}
Moreover, we have
\begin{align}\label{eq:conv2}
\begin{aligned}
\bfJ_n&\rightharpoonup \bfJ=-\frac{\partial\varrho}{\partial\varphi}\nabla\mu\quad\text{in}\quad L^2({\Q_T}),\\
\bfv_n\otimes\bfJ_n&\rightharpoonup \bfv\otimes\bfJ\quad\text{in}\quad L^{s_0}({\Q_T}).
\end{aligned}
\end{align}
for some $s_0>1$.
On account of  \eqref{eq:J} we have
\begin{align}\label{eq:DJv}
\nabla(\bfv_n\otimes\bfJ_n)&\rightharpoonup \nabla(\bfv\otimes\bfJ)\quad\text{in}\quad L^{s_0}({\Q_T}),
\end{align}
combining \eqref{eq:J} with \eqref{eq:convv}.

Combining \eqref{eq:conv1}$_4$ and \eqref{eq:convv} yields $\tilde{\bfH}=\bfv\otimes \varrho\bfv$. This means we also have the convergences
\begin{align}\label{eq:conv3}
\begin{aligned}
\nabla(\bfv_n\otimes\varrho_n\Psi_{1/n}(\bfv_n))&\rightharpoonup \nabla(\bfv\otimes\varrho\bfv)\quad\text{in}\quad L^{p\frac{d+2}{2d+2}}({\Q_T}),\\
\Psi_{1/n}(\mu_n\nabla \varphi_n)&\rightharpoonup P_\sigma(\mu\nabla \varphi)\quad\text{in}\quad L^{1}({\Q_T}),
\end{aligned}
\end{align}
using the pointwise convergence $\Psi_{1/n}\rightarrow P_\sigma$. We have the limit equation
\begin{align}\label{eq:limitvdelta_}
\begin{aligned}
-\int_{\Q_T}\varrho\bfv&\cdot\partial_t\bfeta\dxt +\int_{\Q_T}\tilde\bfS:\nabla\bfeta\dxt-\int_{\Q_T}\bfv\otimes \varrho\bfv:\nabla\bfeta\dxt\\&=\int_{\Q_T} R\tfrac{\bfv}{2}\cdot\bfeta\dxt+\int_{\Q_T}\bfv\otimes\bfJ:\nabla\bfeta\dxt+\int_{\Q_T} \mu\nabla \varphi\cdot\bfeta\dxt
\end{aligned}
\end{align}
for all $\bfeta\in C^\infty_{0,\sigma}({\Q_T})$ and the equation for the difference $\varrho_n\bfv_n-\varrho\bfv$
\begin{align}\label{eq:limitum}
\begin{aligned}
-\int_{\Q_T}&(\varrho_n\bfv_n-\varrho\bfv)\cdot\partial_t\bfeta\dxt +\int_{\Q_T}\big(\bfS(\varphi_n,\bfD\bfv_n)-\tilde{\bfS}\big):\nabla\bfeta\dxt\\&-\int_{\Q_T}\big(\bfv_n\otimes\varrho_n\Psi_{1/n}(\bfv_n)-\bfv\otimes\varrho\bfv\big):\nabla\bfeta\dxt\\
&=\int_{\Q_T} \big(R_n\tfrac{\bfv_n}{2}-R\tfrac{\bfv}{2}\big)\cdot\bfeta\dxt
+\int_{\Q_T}\big(\bfv_n\otimes\bfJ_n-\bfv\otimes\bfJ\big):\nabla\bfeta\dxt\\
&+\int_{\Q_T} \big(\Psi_{1/n}(\mu_n\nabla \varphi_n)-\mu\nabla \varphi\big)\cdot\bfeta\dxt.
\end{aligned}
\end{align}

\textbf{Step 2: pressure function and strong convergence.}\\
Now we define
\begin{align*}
\bfH^n&:=\bfH_1^n+\bfH_2^n,\quad
\bfH_1^n:=\bfS(\varphi_n,\bfD\bfv_n)-\tilde{\bfS},\\ \bfH_2^n&:=-\bfv_n\otimes\varrho_n\Psi_{1/n}(\bfv_n)
+\bfv\otimes\varrho\bfv-\big(\bfv_n\otimes{\bfJ}_n-\bfv\otimes\bfJ\big)
\\&+\nabla\Delta^{-1}_D\big(\Psi_{1/n}(\mu_n\nabla \varphi_n)-\mu\nabla \varphi\big)+\nabla\Delta^{-1}_D\big(R_n\tfrac{\bfv_n}{2}-R\tfrac{\bfv}{2}\big).
\end{align*}
where $\Delta^{-1}_D$ is the solution operator to the Laplace equation on $\Omega$ with respect to zero boundary conditions.
Then we have
\begin{align*}
-\int_{\Q_T}\big(\varrho_n\bfv_n-\varrho\bfv&\big)\cdot\partial_t\bfeta\dxt \nonumber=-\int_{\Q_T}\bfH_1^n:\nabla\bfeta\dxt+\int_{\Q_T}\Div\bfH_2^n\cdot\bfeta\dxt.
\end{align*}
Moreover, we know
\begin{align}
\bfH_1^n&\rightharpoonup 0\quad\text{in}\quad L^{p'}({\Q_T}),\label{eq:convH1}\\
\bfH_2^n,\Div\bfH_2^n&\rightharpoonup 0\quad\text{in}\quad L^{s_0}({\Q_T})\quad \text{for some}\quad s_0>1\label{eq:convH2}
\end{align}
and introduce the pressure in accordance to Theorem \ref{thm:pressure} and Corollary \ref{cor:pressure}. We obtain functions $\pi^n_h$, $\pi^n_{1}$, $\pi^n_{2}$ such that for some $r\in(1,2]$ and any $\Omega'\Subset\Omega$
\begin{align*}
\|\pi_h^n\|_{L^\infty(0,T;L^{r}(\Omega))}&\leq \,c\,\Big(\|\bfH^n\|_{L^r({\Q_T})}+\|\bfu^n\|_{L^\infty(0,T;L^2(\Omega))}\Big),\\
\|\pi_1^n\|_{L^{p'}({\Q_T})}&\leq \,c\,\|\bfH_1^n\|_{L^{p'}({\Q_T})}\\
\|\pi_2^n\|_{L^{r}({\Q_T})}&\leq \,c\,\|\bfH_2^n\|_{L^{r}({\Q_T})}\\
\|\nabla\pi_2^n\|_{L^{r}((0,T)\times\Omega')}&\leq \,c(\Omega')\,\|\nabla\bfH_2^n\|_{L^{r}({\Q_T})}
\end{align*}
for some $c>0$ and
\begin{align}
-\int_{\Q_T}\big(\varrho_n\bfv_n-\varrho\bfv&
\big)\cdot\partial_t\bfeta\dxt+\int_{\Q_T}\pi^n_h\,\partial_t\Div\bfeta\dxt \nonumber\\&=-\int_{\Q_T}\big(\bfH_1^n-\pi_{1}^n I\big):\nabla\bfeta\dxt+\int_{\Q_T}\Div\big(\bfH_2^n-\pi_{2}^n I\big)\cdot\bfeta\dxt\label{eq:diffpi0}
\end{align}
for all $\bfeta \in C_0^\infty(\Q_T)^d$.
We have the following convergences for the pressure functions:
\begin{align}\label{eq:convpi}
\begin{aligned}
\pi_h^n&\rightharpoonup^\ast 0\quad\text{in}\quad L^\infty(0,T;L^{r}(\Omega)),\\
\pi_{1}^n&\rightharpoonup 0\quad\text{in}\quad L^{p'}({\Q_T}),\\
\pi_{2}^n&\rightharpoonup 0\quad\text{in}\quad L^{r}({\Q_T}),\\
\nabla\pi_{2}^n&\rightharpoonup 0\quad\text{in}\quad L^{r}((0,T)\times\Omega')\quad \text{for all}\quad\Omega'\Subset\Omega.
\end{aligned}
\end{align}
A main difference to the approach in \cite{Wolf} is that $\pi_h^n$
is not harmonic. In fact, we have by Theorem \ref{thm:pressure} a)
\begin{align}\label{eq:Deltapih}
\Delta\pi_h^n=-\Div(\varrho_n\bfv_n-\varrho\bfv)
=-\nabla\varrho_n\cdot\bfv_n+\nabla\varrho\cdot\bfv\in L^\infty(0,T;L^2(\Omega))
\end{align}
uniformly in $n$ (recall \eqref{eq:conv1} and \eqref{eq:creg}). Moreover,
\eqref{eq:creg} and \eqref{eq:convv} imply
\begin{align}\label{eq:convpih}
\Delta\pi_h^n&\rightarrow 0 \quad\text{in}\quad L^{s}({\Q_T})\quad \text{for all}\quad s<p\frac{d+2}{d}.
\end{align}
For any $\Omega'\Subset\Omega''\Subset\Omega$ and any $q,\correct{r}\in(1,\infty)$ we have
\begin{align}\label{eq:estpih}
\int_{\Omega'}|\nabla^2\pi_h^n(t)|^q\dx\leq c(q,\Omega',\Omega'')\bigg(\int_{\Omega''}|\Delta\pi_h^n(t)|^{\correct{q}}\dx+\bigg(\int_{\Omega''}|\pi_h^n(t)|^{\correct{r}}\dx\bigg)^{\correct{\frac{q}r}}\bigg).
\end{align}
This estimate is a consequence of \cite[Theorem~9.11]{GilTru} in combination with Sobolev-embeddings.
We gain on account of \eqref{eq:convpi}--\eqref{eq:convpih} and \eqref{eq:estpih}
\begin{align}\label{eq:regpih}
\pi_h^n\in L^\infty(0,T;W^{2,s}_{loc}(\Omega))\quad \text{for all}\quad s<p\frac{d+2}{d},
\end{align}
uniformly in $n$. Now we are concerned with compactness of $\pi_h^n$.
We set $\bfq_n:=\varrho_n\bfv_n-\varrho\bfv+\nabla\pi_h^n$ and have
\begin{align}\label{eq:qspace}
\bfq_n\in L^p(0,T;W^{1,p}_{loc}(\Omega))
\end{align}
uniformly in $n$ (combining \eqref{eq:conv1}, \eqref{eq:creg} and \eqref{eq:regpih}). Moreover, there holds by \eqref{eq:diffpi0}
\begin{align}\label{eq:qtime}
\bfq_n\in W^{1,r}(0,T;W^{-1,r}(\Omega))
\end{align}
uniformly in $n$. This implies by the Aubin-Lions Theorem using \eqref{eq:regpih}
\begin{align}\label{eq:qconv}
\bfq_n\rightarrow 0\quad \text{in}\quad L^q(0,T;L^q_{loc}(\Omega))\quad \text{for all}\quad q<p\frac{d+2}{d}.
\end{align}
Combining this with \eqref{eq:convv} and \eqref{eq:creg}$_4$ yields
\begin{align}\label{eq:qcomp}
\nabla\pi^n_h\rightarrow 0\quad \text{in}\quad L^q(0,T;L^q_{loc}(\Omega))\quad \text{for all}\quad q<p\frac{d+2}{d}.
\end{align}
We can assume that $\int_\Omega \pi_h^n(x,t)\dx=0$ for a.e. $t$, such that
\begin{align}\label{eq:harmpi2}
\pi_h^n&\rightarrow 0 \quad\text{in}\quad L^{q}(0,T;L^{q}_{loc}(\Omega))\quad \text{for all}\quad q<p\frac{d+2}{d}.
\end{align}
In order to show \eqref{eq:harmpi2} we make use of the \Bogovskii\, operator introduced in \cite{Bog80}. It is a solution operator to divergence equation with respect to zero boundary conditions. Taking an arbitrary ball $B\Subset\Omega$ and the \Bogovskii\, operator $\Bog_B$ with respect to this ball we have $\Bog_B :L^r_{(0)}(B)\rightarrow W^{1,r}_0(B)$ for all $r\in(1,\infty)$. So we gain
\begin{align*}
\|\pi_h^n\|_{L^q(B)}&=\sup_{\eta\in C^\infty_0(B),\,\,\|\eta\|_{q'}=1}\int_\Omega \pi_h^n\,(\eta-\eta_B)\dx\\&=\sup_{\eta\in C^\infty_0(B),\,\,\|\eta\|_{q'}=1}\int_\Omega  \pi_h^n\,\Div\Bog_B(\eta-\eta_B)\dx\\
&\leq \sup_{\eta\in C^\infty_0(B),\,\,\|\eta\|_{q'}=1}\int_B  \nabla\pi_h^n\cdot\Bog_B(\eta-\eta_B)\dx\\
&\leq \sup_{\eta\in C^\infty_0(B),\,\,\|\eta\|_{q'}=1}\|\nabla\pi_h^n\|_{L^q(B)}\|\Bog_B(\eta-\eta_B)\|_{L^{q'}(B)}\\
&\leq\,c\,\|\nabla\pi_h^n\|_{L^q(B)}.
\end{align*}
Integrating in time and using \eqref{eq:qcomp} yields \eqref{eq:harmpi2}.
Finally \eqref{eq:convpih}--\eqref{eq:harmpi2} imply
\begin{align}\label{eq:harmpi}
\pi_h^n&\rightarrow 0\quad\text{in}\quad L^{q}(0,T;W^{2,q}_{loc}(\Omega))\quad \text{for all}\quad q<p\frac{d+2}{d}.
\end{align}
In the following we need to show that $\tilde{\bfS}=\bfS(\varphi,\bfD\bfv)$.\\\

\textbf{Step 3: $L^\infty$-truncation and monotone operator theory.}\\
By density arguments we are allowed to test with $\bfeta\in L^p(0,T;W_0^{1,p}(\Omega)^d)\cap L^\infty(\Q_T)^d$ in \eqref{eq:diffpi0}.
Since the function $\bfv$ does not belong to this class, the $L^\infty$-truncation is an appropriate method (see \cite{FrehseMalekSteinhauer1} for the steady case and \cite{Wolf} for the unsteady problem).\\
We define $h_L$ and $H_L$, $L\in\N_0$, by
\begin{align*}
h_L(s)&:=\int_0^s \Upsilon_L(\theta)\theta\,\mathrm{d}\theta,\quad H_L(\bfxi):=h_L(|\bfxi|),\\
\Upsilon_L&:=\sum_{\ell=1}^L \psi_{2^{-\ell}},\quad \psi_\delta(s):=\psi(\delta s),
\end{align*}
where $\psi\in C^\infty([0,\infty))$ with $0\leq\psi\leq1$, $\psi\equiv1$ on $[0,1]$, $\psi=0$ on $[2,\infty)$ and $0\leq-\psi'\leq2$. 
Now we use in \eqref{eq:diffpi0} the test-function $\bfeta=\eta\Upsilon_L(|\bfq_n|)\bfq_n$, where $\eta\in C^\infty_0(\Omega)$.
This yields for a.e. $t\in(0,T)$ using $\bfq_n(0)=0$ (recall Theorem \ref{thm:pressure} a))
\begin{align*}
\int_\Omega \eta H_L(\bfq_{n}(t))\dx
=&\int_0^t\left\langle \partial_t\bfq_n, \eta \Upsilon_L(|\bfq_n|)\bfq_n\right\rangle\ds\\
=&-\int_\Omega\int_0^t\eta\big(\bfH_1^{n}-\pi^{n}_{1}I\big)
:\nabla\big(\Upsilon_L(|\bfq_{n}|)\bfq_{n}\big)\dxs\\
&-\int_\Omega\int_0^t\big(\bfH_1^{n}-\pi^{n}_{1}I\big)
:\nabla\eta\otimes\Upsilon_L(|\bfq_n|)\bfq_{n}\dxs\\
&+\int_\Omega\int_0^t\eta\Div\big(\bfH_2^{n}-\pi^{n}_{2}I\big)\Upsilon_L(|\bfq_n|)\bfq_n\dxs
=:(I)+(II)+(III)
\end{align*}
and $\partial_t \bfq_n \in L^{p'}(0,T;(W^{1,p}(\Omega)\cap L^\infty(\Omega))')$, recall \eqref{eq:qtime}.
The aim of the following observations is to show that the terms $(II)$ and $(III)$ vanish for $n\rightarrow\infty$ which gives the same for $(I)$ (note that the term on the left hand side vanishes for a.e. $t$ by \eqref{eq:qconv}). By monotone operator theory this yields $\bfD\bfv_n\rightarrow\bfD\bfv$ a.e.
Due to the construction of $\Upsilon_L$ we obtain, after passing to a subsequence,
\begin{align}\label{eq:convpsiu}
\Upsilon_L(|\bfq_n|)\bfq_n&\longrightarrow 0\quad\text{in}\quad L^r({\Q_T})\quad \text{as}\quad n\rightarrow\infty
\end{align}
for all $r<\infty$ (first, we have boundedness in $L^r$, then the strong convergence follows from \eqref{eq:qconv}). This implies
\begin{align*}
(II),\,(III)
\longrightarrow0,\quad n\rightarrow\infty,
\end{align*}
as a consequence of \eqref{eq:convH1} and \eqref{eq:convH2}.
Plugging all together, we have shown
\begin{align}\label{eq:monop1}
\begin{aligned}
\limsup_{n\rightarrow\infty}&\,\int_{\Q_t}\eta  \big(\bfS(\varphi_n,\bfD\bfv_n)-\tilde{\bfS})\big):\Upsilon_L(|\bfq_n|)\bfD\bfq_n\dxs\\
&\leq \limsup_{n\rightarrow\infty}\int_{\Q_t}(- \eta )\big(\bfS(\varphi_n,\bfD\bfv_n)-\tilde{\bfS})\big):\nabla\Upsilon_L(|\bfq_n|)\otimes\bfq_n\dxs\\
&+ \limsup_{n\rightarrow\infty}\int_{\Q_t} \eta\,\pi^{n}_{1}\Div\big(\Upsilon_L(|\bfq_{n}|)\bfq_{n}\big)\dxs.
\end{aligned}
\end{align}
Now we want to show that the right-hand-side is bounded in $L\in\N$. Since $\Div \bfq_{n}=0$, there holds
\begin{align*}
\limsup_{n\rightarrow\infty}&\int_{\Q_t} \eta\,\pi^{n}_{1}\Div\big(\Upsilon_L(|\bfq_{n}|)\bfq_n\big)\dxs\\&=\limsup_{n\rightarrow\infty}\int_{\Q_t} \eta\,\pi^{n}_{1}\nabla\Upsilon_L(|\bfq_{n}|)\cdot\bfq_n\dxs.
\end{align*}
So, by \eqref{eq:convpi}, we only need to show
\begin{align}\label{eq:Psip}
\nabla\Upsilon_L(|\bfq_n|)\bfq_n\in L^p(0,T;L^p_{loc}(\Omega))
\end{align}
uniformly in $L$ and $n$ to conclude
\begin{align}\label{eq:monop2}
\limsup_{n\rightarrow\infty}&\,\int_{\Q_t}\eta  \big(\bfS(\varphi_n,\bfD\bfv_n)-\tilde{\bfS})\big):\Upsilon_L(|\bfq_n|)\bfD\bfq_n\dxs\leq K.
\end{align}
We have for all $\ell\in\N_0$
\begin{align*}
\big|\nabla\big\{\psi_{2^{-\ell}}(|\bfq_n|)\big\}\bfq_n\big|&\leq \big|\psi'_{2^{-\ell}}(|\bfq_n|)\bfq_n\otimes\nabla\bfq_n\big|\\
&\leq -2^{-\ell}|\bfq_n|\psi'(2^{-\ell}|\bfq_n|)|\nabla\bfq_n|\leq c |\nabla\bfq_n|\chi_{A_\ell},\quad \text{where}\\
A_\ell&:=\left\{2^{\ell}< |\bfq_n|\leq 2^{\ell+1}\right\}.
\end{align*}
This implies
\begin{align*}
\big|\nabla\Upsilon_{L}(|\bfq_{n}|)\bfq_{n}\big|&\leq \sum_{\ell=0}^L \big|\nabla\big\{\psi_{2^{-\ell}}(|\bfq_{n}|)\big\}\bfq_{n}\big|\leq \tilde{c}\sum_{\ell=0}^L \big|\nabla\bfq_{n}\big|\chi_{A_\ell}\leq c|\nabla\bfq^{n}|.
\end{align*}
This yields \eqref{eq:Psip} and hence also \eqref{eq:monop2} is shown.
Now we consider
\begin{align*}
\Sigma_{L,n}:=\,\int_{\Q_t}\eta  \big(\bfS(\varphi_n,\bfD\bfv_n)-\tilde{\bfS})\big):\Upsilon_L(|\bfq_{n}|)\bfD\bfq_{n}\dxs.
\end{align*}
On account of \eqref{eq:monop2} we have $\Sigma_{L,n}\leq K$ independent of $L$ and $n$. Thus, using Cantor's diagonalizing principle we obtain a subsequence with
\begin{align*}
\sigma_{\ell,n_k}:=\,\int_{\Q_t}\eta  \big(\bfS(\varphi_{n_k},\bfD\bfv_{n_k})-\tilde{\bfS})\big):\psi_{2^{-\ell}}(|\bfq_{n_k}|)\bfD\bfq_{n_k}\dxs
\longrightarrow \sigma_\ell\quad \text{as}\quad l\to \infty
\end{align*}
for all $ \ell\in\N_0.$ We know as a consequence of the monotonicity of $\bfS$ assumed in (A2), $\nabla \rho_n\to \nabla \rho$ in $L^s(\Q_T)$ for any $s<\infty$, and $\nabla^2 \pi^n_h\to 0$ in $L^q_{loc}(\Q_T)$ for $n\rightarrow\infty$, see \eqref{eq:harmpi}, that $\sigma_\ell\geq0$ for all $\ell\in\N$.
Moreover, $\sigma_\ell$ is increasing in $\ell$. This implies on account of (\ref{eq:monop2})
\begin{align*}
0\leq \sigma_0\leq \frac{\sigma_0+\sigma_1+...+\sigma_\ell}{\ell}\leq \frac{K}{\ell}
\end{align*}
for all $\ell\in\N$. Hence we have $\sigma_0=0$ and therefore
\begin{align*}
\int_{\Q_T} \eta\big(\bfS(\varphi_n,\bfD\bfv_n)-\tilde{\bfS})\big):\psi_1(|\bfq_{n}|)\bfD\bfq_{n}\dxt\longrightarrow0,\quad n\rightarrow\infty.
\end{align*}
Due to \eqref{eq:conv1}, \eqref{eq:creg}, \eqref{eq:convv} uniform boundedness of $\psi_1(|\bfq_n|)$ and \eqref{eq:harmpi} we have for $n\rightarrow\infty$
\begin{align*}
&\int_{\Q_t} \eta\big(\bfS(\varphi_n,\bfD\bfv_n)-\bfS(\varphi,\bfD\bfv)\big):\psi_1(|\bfq_{n}|)\bfD\nabla\pi_h^n\dxs\longrightarrow0,\\
&\int_{\Q_t} \eta\big(\bfS(\varphi_n,\bfD\bfv_n)-\bfS(\varphi,\bfD\bfv)\big):\psi_1(|\bfq_{n}|)\bfD\big((\varrho_n-\varrho)\bfv_n\big)\dxs\longrightarrow0,
\end{align*}
and hence
\begin{align*}
\int_{\Q_t} \eta\big(\bfS(\varphi_n,\bfD\bfv_n)-\bfS(\varphi,\bfD\bfv)\big):\psi_1(|\bfq_{n}|)\bfD(\varrho(\bfv_{n}-\bfv))\dxs\longrightarrow0\quad \text{as } n\rightarrow\infty.
\end{align*}
Using $\nabla\varrho\in L^\infty(\Q_T)$ \eqref{eq:conv1} and \eqref{eq:convv} this implies
\begin{align}\label{eq:monop1b}
\int_{\Q_t} \varrho\eta\big(\bfS(\varphi_n,\bfD\bfv_n)-\bfS(\varphi,\bfD\bfv)\big):\psi_1(|\bfq_{n}|)\bfD(\bfv_{n}-\bfv)\dxs\longrightarrow0\quad \text{as } n\rightarrow\infty.
\end{align}
For $\vartheta\in(0,1)$ we obtain
\begin{align*}
&\int_{\Q_t} \Big(\eta\varrho\big(\bfS(\varphi_n,\bfD\bfv_{n})-\bfS(\varphi,\bfD\bfv)\big):\bfD(\bfv_{n}-\bfv)\Big)^\vartheta\dxs\\&=\int_{{\Q_t}} \chi_{\set{|\bfq_{n}|> 1}}\Big(\eta\varrho\big(\bfS(\varphi_n,\bfD\bfv_{n})-\bfS(\varphi,\bfD\bfv)\big):\bfD(\bfv_{n}-\bfv)\Big)^\vartheta\dxs\\
&+\int_{{\Q_t}} \chi_{\set{|\bfq_{n}|\leq 1}}\Big(\eta\varrho\big(\bfS(\varphi,\bfD\bfv_n)-\bfS(\varphi,\bfD\bfv)\big):\bfD(\bfv_{n}-\bfv)\Big)^\vartheta\dxs\\
&=:(A)+(B).
\end{align*}
By \eqref{eq:qconv} and \eqref{eq:conv1} there holds
\begin{align*}
(A)&\leq \LL^{d+1}\Big(\Big\{(t,x) \in (0,T)\times\mathrm{spt}(\eta):|\bfq_{n}(t,x)|\geq 1\Big\}\Big)^{1-\vartheta}\\
&\times\bigg(\int_{ {\Q_t}}\eta\varrho\big(\bfS(\varphi_n,\bfD\bfv_n)-\bfS(\varphi,\bfD\bfv)\big):\bfD(\bfv_{n}-\bfv)\dxs\bigg)^\vartheta\\
&\leq c\bigg(\int_{0}^T\int_{\mathrm{spt}(\eta)} |\bfq_n|^2\dx\ds\bigg)^{1-\vartheta}\longrightarrow0\quad \text{as } n\rightarrow\infty,
\end{align*}
where we took into account H\"older's inequality.  Since $(B)$ also vanishes for $n\rightarrow\infty$ by \eqref{eq:monop1b}, we
finally have shown
\begin{align*}
\int_{\Q_t} \Big(\eta\varrho\big(\bfS(\varphi_n,\bfD\bfv_n)-\bfS(\varphi,\bfD\bfv)\big):\bfD\big(\bfv_n-\bfv\big)\Big)^\vartheta\dxs
&\longrightarrow0\quad \text{as } n\rightarrow\infty,
\end{align*}
for all $\vartheta\in(0,1)$. We have by assumption (A2) choosing $\vartheta>\frac{1}{2}$
\begin{align*}
\int_{\Q_t} &\Big|\eta\varrho\big(\bfS(\varphi_n,\bfD\bfv_n)-\bfS(\varphi,\bfD\bfv_n)\big):\bfD\big(\bfv_n-\bfv\big)\Big|^\vartheta\dxs\\&\leq\,c\,\int_{\Q_t} |\varphi_n-\varphi|^\vartheta(1+|\bfD\bfv_n|^{\vartheta(p-1)})\big(|\bfD\bfv_n|^\vartheta+|\bfD\bfv|^\vartheta\big)\dxs\\
&\leq\,c\,\bigg(\int_{\Q_T}|\varphi_n-\varphi|^{\frac{\vartheta}{1-\vartheta}}\dxt\bigg)^{1-\vartheta}\bigg(\int_{\Q_T} (1+|\bfD\bfv_n|^{p}+|\bfD\bfv|^p)\dxs\bigg)^\vartheta
\longrightarrow_{n\to\infty}0,
\end{align*}
where we have used \eqref{eq:creg} and \eqref{eq:conv1}$_1$. This finally implies
\begin{align*}
\int_{\Q_t} \Big(\eta\varrho\big(\bfS(\varphi,\bfD\bfv_n)-\bfS(\varphi,\bfD\bfv)\big):\bfD\big(\bfv_n-\bfv\big)\Big)^\vartheta\dxs
&\longrightarrow0\quad \text{as } n\rightarrow\infty,
\end{align*}
for all $\vartheta\in\big(\frac{1}{2},1\big)$ and a.e. $t\in(0,T)$. The monotonicity of $\bfS$ supposed in (A2) implies that $\bfD\bfv_n\rightarrow\bfD\bfv$ a.e. as $\eta$ is arbitrary and $\varrho$ is strictly positive.
This justifies the limit procedure in the energy integral, e.g. $\tilde{\bfS}=\bfS(\varphi,\bfD\bfv)$ is shown and the proof of Theorem \ref{thm:main} is therefore complete.

\appendix 
\section{Appendix: Pressure Decomposition}
The next theorem is in the spirit of \cite[Theorem 2.6]{Wolf} but without the condition
$\Div\bfu=0$. 
\begin{theorem}\label{thm:pressure}
Let $\bfu\in C_w([0,T];L^2(\Omega)^d)$, $1<q<\infty$, and let $\bfH\in L^q({\Q_T})^{d\times d}$ such that
\begin{align*}
\int_{\Q_T} \bfu\cdot\partial_t\bfeta\dxt=\int_{\Q_T}\bfH:\nabla\bfeta\dxt
\end{align*}
for all $\bfeta\in C^\infty_{0,\sigma}({\Q_T})$.
Then there are integrable functions $\pi_h$ and $\pi_0$ such that
\begin{align*}
\int_{\Q_T} \bfu\cdot\partial_t\bfeta\dxt=\int_{\Q_T}\bfH:\nabla\bfeta\dxt+\int_{\Q_T}\pi_0\cdot\Div\bfeta\dxt+\int_{\Q_T}\pi_h\cdot\partial_t\Div\bfeta\dxt
\end{align*}
for all $\bfeta\in C^\infty_0({\Q_T})$. 
\begin{itemize}
\item[a)] There holds $\pi_h\in C_w([0,T],L^{\min\set{2,q}}(\Omega))$, $\pi_h(0)=0$ and $\Delta\pi_h=-\Div(\bfu-\bfu_0)$.
\item[b)] We have for any $1<r\leq 2$
\begin{align}\label{eq:pih}
\|\pi_h\|_{L^\infty(0,T;L^r(\Omega))}&\leq \,c(r,q)\,\Big(\|\bfH\|_{L^r({\Q_T})}+\|\bfu\|_{L^\infty(0,T;L^2(\Omega))}\Big),\\
\|\pi_0\|_{L^q({\Q_T})}&\leq \,c(q)\,\|\bfH\|_{L^q({\Q_T})}\label{eq:pi0}
\end{align}
for some $c(r,q),c(q)>0$ independent of $\bfu,\bfH,\pi_h,\pi_0$.
\end{itemize}
\end{theorem}
\begin{corollary}\label{cor:pressure}
Let the assumptions of Theorem \ref{thm:pressure} be satisfied. Assume in addition
that $\bfH=\bfH_1+\bfH_2$ where $\bfH_1\in L^{q_1}({\Q_T})^{d\times d}$ and $\bfH_2\in L^{q_2}({\Q_T})^{d\times d}$ for some $1<q_1,q_2<\infty$. 
\begin{itemize}
\item[a)] Then we have $\pi_0=\pi_1+\pi_2$, where
\begin{align}
\|\pi_1\|_{L^{q_1}({\Q_T})}&\leq \,c(q_1,q_2)\,\|\bfH_1\|_{L^{q_1}({\Q_T})},\label{eq:pi1}\\
\|\pi_2\|_{L^{q_2}({\Q_T})}&\leq \,c(q_1,q_2)\,\|\bfH_2\|_{L^{q_2}({\Q_T})}\label{eq:pi2}
\end{align}
for some constant $c(q_1,q_2)>0$.
\item[b)] If  $\nabla\bfH_i\in L^{q_i}({\Q_T})$ for some $i\in \{1,2\}$, then we have for all $\Omega'\Subset\Omega$
\begin{align}
\|\nabla\pi_i\|_{L^{q_2}((0,T)\times\Omega')}&\leq \,c(q_1,q_2)\,\|\nabla\bfH_i\|_{L^{q_i}({\Q_T})}\label{eq:pi2'}
\end{align}
for $i=1,2$ and some constant $c(q_1,q_2)>0$.
\end{itemize}
\end{corollary}
\noindent
\emph{Proof of Theorem \ref{thm:pressure}:}
Following \cite{Wolf} (proof of Thm. 2.6) there is some $\tilde\pi\in C_w([0,T];L^r(\Omega))$, where $r=\min\{2,q\}$, with $\tilde\pi(0)=0$ such that
\begin{align*}
\int_{\Q_T}\tilde\pi\,\Div\bfeta\dxt=\int_{\Q_T} \big(\bfu-\bfu(0)\big)\cdot\bfeta\dxt+\int_{\Q_T} \int_0^t\bfH\,\dd\sigma:\nabla\bfeta\dxt
\end{align*}
for all $\bfeta\in C^\infty_0({\Q_T})^d$ and
\begin{align*}
\|\tilde\pi\|_{L^\infty(0,T;L^r(\Omega))}&\leq \,c\,\Big(\|\bfH\|_{L^r({\Q_T})}+\|\bfu\|_{L^\infty(0,T;L^2(\Omega))}\Big)
\end{align*}
for some $c>0$.
Let $\mathcal L$ be the solution operator to the bi-Laplace equation with zero boundary values for the solution and its gradient. Then we have that $\mathcal L$ extends to a bounded linear operator
\begin{align}\label{eq:bi}
\mathcal L:W^{-2,q}(\Omega)\rightarrow W^{2,q}_0(\Omega)
\end{align}
for all $q\in(1,\infty)$, see \cite{Mul}. We decompose
$\tilde\pi=\tilde\pi_0+\pi_h$, where
\begin{align*}
\tilde\pi_0(t)=\Delta\mathcal L F(t), \quad \langle F(t),\eta\rangle = \int_\Omega\int_0^t\bfH(x,\sigma)\ds: \nabla^2\eta(x) \dx\quad \text{for all}\quad\eta\in W^{2,q'}_0(\Omega)
\end{align*}
and $\pi_h=\tilde\pi-\tilde\pi_0$. We have $\tilde\pi_0(0)=0$ and hence $\pi_h(0)=0$. Due to \eqref{eq:bi}
we have \eqref{eq:pih}. Moreover, there
holds
\begin{align*}
\int_{\Q_T}\pi_h\,\Delta\eta\dxt=\int_{\Q_T}(\bfu-\bfu(0))\cdot\nabla\eta\dxt
\end{align*}
for all $\eta\in C^\infty_0(Q)$ and so $\Delta\pi_h=-\Div(\bfu-\bfu(0))$ a.e.
We set $\pi_0=\partial_t\tilde\pi_0$ such that
\begin{align*}
\int_\Omega \pi_0\,\Delta\eta\dx=\int_\Omega \bfH:\nabla^2\eta\dx
\end{align*}
for all $\eta\in C^\infty_0(\Omega)$ and a.e. in $t$. We obtain \eqref{eq:pi0}
as a consequence of \eqref{eq:bi}.\qed 

\noindent
\emph{Proof of Corollary \ref{cor:pressure}.}
We recall from the proof of Theorem \ref{thm:pressure} that $\pi_0=\Delta\mathcal L F$. So we set
\begin{align*}
\pi_1=\Delta\mathcal LF_1,\quad \pi_2=\Delta\mathcal LF_2,
\end{align*}
where $F_1,F_2$ are defined analoguously to $F$.
The claim follows from the continuity properties of the operator $\Delta\mathcal L$ (which follow from \eqref{eq:bi})
and local regularity theory for the bi-Laplace equation. 
\hfill$\Box$

\section{Appendix: Evolution Equations for Monotone Operators}\label{app:EvolEq}

In the following we will recall some results and definitions from the theory of monotone operators, subgradients and an associated evolution equation. Furthermore, we prove a characterization of the operator $\A$ from the proof of Lemma~\ref{lem:ExistenceCH} as a subgradient. For an introduction to the theory of monotone operators we refer to Br\'ezis~\cite{Brezis} and Showalter~\cite{Showalter}. In the following let $H$ be a real-valued and separable Hilbert space. 
Recall that $\A\colon H\to \mathcal{P}(H)$ is a monotone operator if 
\begin{equation*}
  (w-z,x-y)_H \geq 0 \qquad \text{for all}\ w\in \A(x), z\in \A (y).
\end{equation*}
and $\mathcal{D}(A)=\{x\in H: \A(x)\neq \emptyset\}$.
If $\varphi\colon H\to \R\cup \{+\infty\}$ is a convex function, then $\dom (\varphi)= \{x\in H:\varphi (x)<\infty\}$. Moreover, $\varphi$ is called proper if $\dom (\varphi)\neq \emptyset$. The subgradient $\partial_H\varphi\colon H\to \mathcal{P}(H)$ (with respect to $H$) is defined by $w\in \partial_H \varphi(x)$ if and only if 
\begin{equation*}
    \varphi(\xi)\geq \varphi(x)+ (w,\xi-x)_H \qquad \text{for all}\ \xi \in H.
\end{equation*}
$\partial_H\varphi$ is a monotone operator on $H$. If $\varphi$ is lower semi-continuous, then $\partial\varphi$ is maximal monotone, cf. \cite[Exemple 2.3.4]{Brezis}.

\begin{theorem}\label{thm:MonotoneLipschitz}
  Let $H_0, H_1$ be real-valued, separable Hilbert spaces such that $H_1 \hookrightarrow H_0$ densely. Moreover, let $\varphi\colon H_0 \to \R\cup \{+\infty\}$ be a proper, convex and lower semi-continuous functional such that $\varphi=\varphi_1+\varphi_2$, where $\varphi_2\geq 0$ is convex and lower semi-continuous, $\dom\varphi_1 = H_1$, and $\varphi_1|_{H_1}$ is a bounded, coercive, quadratic form on $H_1$. Set $\A= \partial_H\varphi$. 
Furthermore, assume that $\B\colon[0,T]\times H_1\to H_0$ is measurable in $t\in [0,T]$ and Lipschitz continuous in $v\in H_1$ satisfying 
\begin{equation*} 
  \|B(t,v_1)- B(t,v_2)\|_{H_0} \leq M(t) \|v_1-v_2\|_{H_1} \quad\text{for a.e.}\ t\in [0,T],
\end{equation*}
for all $v_1,v_2 \in H_0$, and for some $M\in L^2(0,T)$.
  Then for every $u_0\in \dom (\varphi)$ and $f\in L^2(0,T;H_0)$ there is a unique $u\in W^{1,2}(0,T;H_0)\cap L^\infty(0,T;H_1)$ with $u(t)\in \mathcal{D}(\A)$ for a.e. $t>0$ solving 
  \begin{eqnarray}\label{eq:1}
    \frac{du}{dt} (t) +  \A(u(t)) &\ni&  \B(t,u(t)) + f(t) \quad \text{for a.a.}\ t\in (0,T), \\
    u(0) &=& u_0.
  \end{eqnarray}
  Moreover, $\varphi(u)\in L^\infty(0,T)$.
\end{theorem}
We refer to \cite[Theorem 4]{ModelH} for the proof.

In order to apply the latter theorem we will use:
\begin{lemma} 
  Let $E_0$ be as in Section~\ref{sec:Approx}. Then the subgradient  $\partial_{L^2_{(0)}} E_0$ of $E_0$ is single-valued and we have
  \begin{alignat*}{1}
    \mathcal{D}(\partial_{L^2_{(0)}} E_0) &=\{ u\in H^2(\Omega)\cap  H^1_{(0)}(\Omega): \partial_{\mathcal{N}}\varphi|_{\partial\Omega} =0\},\\
\partial_{L^2_{(0)}} E_0(\varphi)&=-\Delta \varphi +P_0f'_0(\varphi)\quad \text{for all}\quad\varphi \in \mathcal{D}(\partial_{L^2_{(0)}} E_0).
  \end{alignat*}
\end{lemma}
\begin{proof}
 The inclusion  $$\{ u\in H^2(\Omega)\cap  H^1_{(0)}(\Omega): \partial_{\mathcal{N}}\varphi|_{\partial\Omega} =0\}\subseteq \mathcal{D}(\partial_{L^2_{(0)}} E_0)$$ and $-\Delta \varphi +f'_0(\varphi)\in\partial_{L^2_{(0)}} E_0(\varphi)$ can be shown in a straightforward manner using the definition of the subgradient.  Conversely, if $w\in \partial_{L^2_{(0)}} E_0(\varphi)$ for some $\varphi\in     \mathcal{D}(\partial_{L^2_{(0)}} E_0)$, then
  \begin{equation*}
    \int_\Omega w(\eta-\varphi) \dx \leq E_0(\eta)-E_0(\varphi)\qquad \text{for all}\quad\eta \in \operatorname{dom}(E_0) = H^1_{(0)}(\Omega).
  \end{equation*}
  Using this inequality with $\eta= \varphi +t\psi$ for $\psi\in H^1_{(0)}(\Omega)$, $t>0$ arbitrary, dividing by $t$, and passing to the limit $t\to 0^+$, one obtains
  \begin{equation*}
    \int_\Omega w \psi \dx \leq \int_\Omega \nabla \varphi\cdot \nabla \psi\dx+\int_\Omega f_0'(\varphi) \psi\dx
  \end{equation*}
for all $\psi\in H^1_{(0)}(\Omega)$. Replacing $\psi$ by $-\psi$ one even obtains equality. Hence $\varphi$ is a weak solution of the Laplace equation with Neumann boundary conditions
\begin{alignat*}{2}
  -\Delta \varphi &= w- P_0 f_0'(\varphi)&\quad &\text{in }\Omega,\\
  \partial_{\mathcal{N}} \varphi &= 0&\quad &\text{on }\partial\Omega,
\end{alignat*}
where $f'_0(\varphi)\in L^2(\Omega)$ due to $\varphi \in L^6(\Omega)$ and (A1).
By standard elliptic regularity one obtains $\varphi \in H^2(\Omega)$ and $\partial_{\mathcal{N}} \varphi|_{\partial\Omega}=0$ in the trace sense. Hence we obtain the converse inclusion for $\mathcal{D}(\partial_{L^2_{(0)}} E_0)$ and $\partial_{L^2_{(0)}} E_0(\varphi)=\{-\Delta \varphi +P_0 f_0'(\varphi)\}$. Therefore the claim is proved.
\end{proof}
\begin{corollary}\label{cor:Subgradient}
We have
  \begin{alignat*}{1}
    \mathcal{D}(\partial_{H^{-1}_{(0)}} E_0) &=\{ u\in H^3(\Omega)\cap  H^1_{(0)}(\Omega): \partial_{\mathcal{N}}\varphi|_{\partial\Omega} =0\},\\
\partial_{H^{-1}_{(0)}} E_0(\varphi)&=-\Delta_N(-\Delta \varphi +f'_0(\varphi))\quad \text{for all}\quad\varphi \in \mathcal{D}(\partial_{H^{-1}_{(0)}} E_0)
  \end{alignat*}  
\end{corollary}
\begin{proof}
The statement is proved in the same way as in the proof of \cite[Corollary~4.4]{AsymptoticCH}, where one uses that $-\Delta \varphi +f'_0(\varphi)\in H^1(\Omega)$ and elliptic regularity theory again.
\end{proof}

\subsection*{Acknowledgement}
The authors thank the referee for a careful reading of the manuscript and the valuable suggestions which significantly improved the presentation of the final version.


\begin{thebibliography}{10}

\bibitem{ModelH}
H.~Abels.
\newblock On a diffuse interface model for two-phase flows of viscous,
  incompressible fluids with matched densities.
\newblock {\em Arch. Rat. Mech. Anal.}, 194(2):463--506, 2009.

\bibitem{AbelsDepnerGarcke}
H.~Abels, D.~Depner, and H.~Garcke.
\newblock Existence of weak solutions for a diffuse interface model for
  two-phase flows of incompressible fluids with different densities.
\newblock {\em J. Math. Fluid Mech.}, 15(3):453--480, 2013.

\bibitem{NonNewtonianModelH}
H.~Abels, L.~Diening, and Y.~Terasawa.
\newblock Existence of weak solutions for a diffuse interface model of
  non-{N}ewtonian two-phase flows.
\newblock {\em Nonlinear Anal. Real World Appl.}, 15:149--157, 2014.

\bibitem{AbelsGarckeGruen2}
H.~Abels, H.~Garcke, and G.~Gr{\"{u}}n.
\newblock Thermodynamically consistent, frame indifferent diffuse interface
  models for incompressible two-phase flows with different densities.
\newblock {\em Math. Models Methods Appl. Sci.}, 22(3):1150013 (40 pages),
  2012.

\bibitem{AsymptoticCH}
H.~Abels and M.~Wilke.
\newblock Convergence to equilibrium for the {C}ahn-{H}illiard equation with a
  logarithmic free energy.
\newblock {\em Nonlinear Anal.}, 67(11):3176--3193, 2007.

\bibitem{AbelsDepnerGarckeDegMob}
Helmut Abels, Daniel Depner, and Harald Garcke.
\newblock On an incompressible {N}avier-{S}tokes/{C}ahn-{H}illiard system with
  degenerate mobility.
\newblock {\em Ann. Inst. H. Poincar\'e Anal. Non Lin\'eaire},
  30(6):1175--1190, 2013.

\bibitem{Amann}
H.~Amann.
\newblock {\em Linear and Quasilinear Parabolic Problems, Volume 1: Abstract
  Linear Theory}.
\newblock Birkh{\"{a}}user, Basel - Boston - Berlin, 1995.

\bibitem{AstaritaMarucci}
G.~Astarita and G.~Marrucci.
\newblock {\em Dynamics of polymeric liquids, volume 1: fluid mechanics}.
\newblock McGraw-Hill, London - New York, 1974.

\bibitem{BirdArmstrongHassager}
R.~Bird, R.~Armstrong, and O.~Hassager.
\newblock {\em Principles of non-Newtonian fluid mechanics}.
\newblock John Wiley, London - New York, 1987.

\bibitem{Bog80}
M.~E. Bogovski{\u\i}.
\newblock Solutions of some problems of vector analysis, associated with the
  operators ${\rm div}$\ and ${\rm grad}$.
\newblock In {\em Theory of cubature formulas and the application of functional
  analysis to problems of mathematical physics (Russian)}, pages 5--40, 149.
  Akad. Nauk SSSR Sibirsk. Otdel. Inst. Mat., Novosibirsk, 1980.

\bibitem{Bosia}
S.~Bosia.
\newblock Analysis of a {C}ahn-{H}illiard-{L}adyzhenskaya system with singular
  potential.
\newblock {\em J. Math. Anal. Appl.}, 397(1):307--321, 2013.

\bibitem{BreitDieningSchwarzacher}
D.~Breit, L.~Diening, and S.~Schwarzacher.
\newblock Solenoidal {L}ipschitz truncation for parabolic {PDE}s.
\newblock {\em Mathematical Models and Methods in Applied Sciences},
  23(14):2671--2700, 2013.

\bibitem{Brezis}
H.~Br{\'e}zis.
\newblock {\em Op\'erateurs maximaux monotones et semi-groupes de contractions
  dans les espaces de {H}ilbert}.
\newblock North-Holland Publishing Co., Amsterdam, 1973.
\newblock North-Holland Mathematics Studies, No. 5. Notas de Matem\'atica (50).

\bibitem{Bu}
M.~Bul\'i\v{c}ek, J.~M\'alek, and K.~R. Rajagopal.
\newblock Navier's slip and evolutionary navier-stokes-like systems with
  pressure and shear-rate dependent viscosity.
\newblock {\em Indiana Univ. Math. J.}, 56:61--86, 2007.

\bibitem{DenkHieberPruessMZ}
R.~Denk, M.~Hieber, and J.~Pr{\"u}ss.
\newblock Optimal {$L^p$}-{$L^q$}-estimates for parabolic boundary value
  problems with inhomogeneous data.
\newblock {\em Math. Z.}, 257(1):193--224, 2007.

\bibitem{DieningMalekSteinhauer}
L.~Diening, J.~M\'alek, and M.~Steinhauer.
\newblock On {L}ipschitz truncations of {S}obolev functions (with variable
  exponent) and their selected applications.
\newblock {\em ESAIM: Control, Optimisation and Calculus of Variations},
  14(2):211--232, 3 2008.

\bibitem{DieningRuzickaWolf}
L.~Diening, M.~R{\r{u}}{\v{z}}i{\v{c}}ka, and J.~Wolf.
\newblock Existence of weak solutions for unsteady motions of generalized
  {N}ewtonian fluids.
\newblock {\em Ann. Sc. Norm. Super. Pisa Cl. Sci. (5)}, 9(1):1--46, 2010.

\bibitem{FeNoPe}
E.~Feireisl, A.~Novotn\'y, and H.~Petzeltov\'a.
\newblock On the existence of globally defined weak solutions to the
  {N}avier-{S}tokes equations.
\newblock {\em Journal of Mathematical Fluid Mechanics}, 3(4):358--392, 2001.

\bibitem{Fe}
E.~Feireisl, L.~Yong, and J.~M\'alek.
\newblock On pde analysis of flows of quasi-incompressible fluids.
\newblock {\em Z. Angew. Math. Mech.}, 96(4):491--508, 2016.

\bibitem{FrehseMalekSteinhauer1}
J.~Frehse, J.~M{\'a}lek, and M.~Steinhauer.
\newblock An existence result for fluids with shear dependent viscosity--steady
  flows.
\newblock {\em Nonlinear Anal.}, 30:3041--3049, 1997.

\bibitem{Fr}
J.~Frehse, J.~M\'alek, and M.~Steinhauer.
\newblock On existence results for fluids with shear dependent
  viscosity--unsteady flows.
\newblock In {\em Partial differential equations (Praha, 1998)}, volume 406 of
  {\em Chapman \& Hall/CRC Res. Notes Math.}, pages 121--129. Chapman \&
  Hall/CRC, Boca Raton, FL, 2000.

\bibitem{FrehseMalekSteinhauer2}
J.~Frehse, J.~M{\'a}lek, and M.~Steinhauer.
\newblock On analysis of steady flows of fluids with shear-dependent viscosity
  based on the {L}ipschitz truncation method.
\newblock {\em SIAM J. Math. Anal.}, 34(5):1064--1083 (electronic), 2003.

\bibitem{GilTru}
D.~Gilbarg and N.~S. Trudinger.
\newblock {\em Elliptic partial differential equations of second order}.
\newblock Classics in Mathematics. Springer-Verlag, Berlin, 2001.
\newblock Reprint of the 1998 edition.

\bibitem{GrasselliPrazakNonNewtonianDIM}
M.~Grasselli and D.~Pra{\v{z}}{\'a}k.
\newblock Longtime behavior of a diffuse interface model for binary fluid
  mixtures with shear dependent viscosity.
\newblock {\em Interfaces Free Bound.}, 13(4):507--530, 2011.

\bibitem{Ladyzenskaja2}
O.~A. Lady\v{z}enskaja.
\newblock {\em The mathematical theory of viscous incompressible flow}.
\newblock Gorden and Breach, 1969.

\bibitem{La1}
O.~A. Lady{\v{z}}enskaja.
\newblock New equations for the description of the motions of viscous
  incompressible fluids, and global solvability for their boundary value
  problems.
\newblock {\em Trudy Mat. Inst. Steklov.}, 102:85--104, 1967.

\bibitem{La2}
O.~A. Lady{\v{z}}enskaja.
\newblock Modifications of the {N}avier-{S}tokes equations for large gradients
  of the velocities.
\newblock {\em Zap. Nau\v cn. Sem. Leningrad. Otdel. Mat. Inst. Steklov.
  (LOMI)}, 7:126--154, 1968.

\bibitem{Lions}
J.-L. Lions.
\newblock {\em {Quelques m\'ethodes de r\'esolution des probl\`emes aux limites
  non lin\'eaires}}.
\newblock {Etudes mathematiques. Paris: Dunod; Paris: Gauthier-Villars. XX, 554
  p.}, 1969.

\bibitem{MNRR}
J.~M{\'a}lek, J.~Ne{\v{c}}as, M.~Rokyta, and M.~R{\r{u}}{\v{z}}i{\v{c}}ka.
\newblock {\em Weak and measure-valued solutions to evolutionary {P}{D}{E}s}.
\newblock Chapman \& Hall, London, 1996.

\bibitem{Mul}
R.~M{\"u}ller.
\newblock Das schwache {D}irichletproblem in {$L^q$} f\"ur den
  {B}ipotentialoperator in beschr\"ankten {G}ebieten und in {A}u\ss engebieten.
\newblock {\em Bayreuth. Math. Schr.}, (49):115--211, 1995.
\newblock Dissertation, Universit{\"a}t Bayreuth, Bayreuth, 1994.

\bibitem{Showalter}
R.~E. Showalter.
\newblock {\em Monotone Operators in {B}anach Space and Nonlinear Partial
  Differential Equations}, volume~49 of {\em Mathematical Surveys and
  Monographs}.
\newblock American Mathematical Society, Providence, RI, 1997.

\bibitem{Triebel1}
H.~Triebel.
\newblock {\em Interpolation Theory, Function Spaces, Differential Operators}.
\newblock North-Holland Publishing Company, Amsterdam, New York, Oxford, 1978.

\bibitem{Wolf}
J.~Wolf.
\newblock Existence of weak solutions to the equations of non-stationary motion
  of non-newtonian fluids with shear rate dependent viscosity.
\newblock {\em Journal of Mathematical Fluid Mechanics}, 9(1):104--138, 2007.

\end{thebibliography}

\begin{thebibliography}{[MMM]}


\bibitem[AM]{AM} G. Astarita and G. Marrucci (1974): Principles of non-Newtonian fluid mechanics. McGraw-Hill, London-New York.
	\bibitem[BAH]{BAH} R. Bird, R. Armstrong, O. Hassager (1987): Dynamics of polymeric liquids, volume 1: fluid mechanics (second edition). John Wiley.	
	\bibitem[Bog]{Bog}
M.~E. Bogovski{\u\i}.
\newblock Solutions of some problems of vector analysis, associated with the
  operators ${\rm div}$\ and ${\rm grad}$.
\newblock In {\em Theory of cubature formulas and the application of functional
  analysis to problems of mathematical physics (Russian)}, pages 5--40, 149.
  Akad. Nauk SSSR Sibirsk. Otdel. Inst. Mat., Novosibirsk, 1980.
\bibitem[BrDS]{BrDS} D. Breit, L. Diening, S. Schwarzacher (2013): Solenoidal Lipschitz truncation for 
parabolic PDE's. Math. Mod. Meth. Appl. Sci. 23, 2671--2700.
	\bibitem[DMS]{DMS} L. Diening, J. M\'alek, M. Steinhauer (2006): On Lipschitz Truncations of Sobolev Functions (with Variable Exponent) and Their Selected Applications. ESAIM Control Optim. Calc. Var. 14 (2008), no. 2,
211--232.
\bibitem[DRW]{DRW} L. Diening, M. R\r{u}\v{z}i\v{c}ka, J. Wolf (2010): Existence of weak solutions for unsteady motions of generalized Newtonian fluids. Ann. Sc. Norm. Sup. Pisa Cl. Sci. (5) Vol. IX (2010), 1-46.
\bibitem[FMS1]{FMS1} J. Frehse, J. M\'alek, and M. Steinhauer (1997): An existence result for fluids with shear dependent
viscosit--steady flows. Nonlinear Anal. 30, pp. 3041--3049.
		\bibitem[FMS2]{FMS2} J. Frehse, J. M\'alek, M. Steinhauer (2003): On analysis of steady flows of fluids with shear-dependent viscosity based on the Lipschitz truncation method. SIAM J. Math. Anal. 34 (5), 1064-1083 (electronic).
 \bibitem[Ga1]{Ga1} G. Galdi (1994): An introduction to the mathematical theory of the Navier-Stokes equations Vol. I, Springer Tracts in Natural Philosophy Vol. 38. Springer, Berlin-New York.
   \bibitem[Ga2]{Ga2} G. Galdi (1994): An introduction to the mathematical theory of the Navier-Stokes equations Vol. II, Springer Tracts in Natural Philosophy Vol. 39. Springer, Berlin-New York.
   \bibitem[La1]{La} O. A. Ladyzhenskaya (1969): The mathematical theory of viscous incompressible flow. Gorden and Breach.
   \bibitem[La2]{La2} O. A. Ladyzhenskaya (1967): On some new equations describing dynamics of incompressible fluids
and on global solvability of boundary value problems to these equations. Trudy Steklov's
Math. Institute 102, 85-104.
   \bibitem[La3]{La3} O. A. Ladyzhenskaya (1968): On some modifications of the Navier-Stokes equations for large gradients
of velocity, Zap. Nauchn. Sem. Leningrad. Otdel. Mat. Inst. Steklov (LOMI) 7, 126-154.
\bibitem[Li]{Li}J. L. Lions (1969): Quelques m\'{e}thodes de r\'{e}solution des probl\`{e}mes aux limites non lin\'{e}aires.
\bibitem[Mul]{Mul} R.~M{\"u}ller, \emph{Das schwache {D}irichletproblem in {$L^q$} f{\"u}r den {B}ipotentialoperator in beschr{\"a}nkten {G}ebieten und in {A}u\ss engebieten}, Bayreuth. Math. Schr. (1995), no.~49, 115--211, Dissertation, Universit{\"a}t Bayreuth, Bayreuth, 1994.
   \bibitem[MNRR]{MNRR} J. M{\'a}lek, J. Ne\v{c}as, M. Rokyta, M. R\r{u}\v{z}i\v{c}ka (1996):                     Weak and measure valued solutions to evolutionary PDEs. Chapman \& Hall, London-Weinheim-New York.
\bibitem[Wo]{Wo} J. Wolf (2007): Existence of weak solutions to the equations of nonstationary motion of non-
Newtonian fluids with shear-dependent viscosity. J. Math. Fluid Mech. 9, 104-138.
\end{thebibliography}

\end{document}